\documentclass{amsart}
\usepackage{amsmath,amssymb,amscd}

\newcommand{\tori}{T^2_{\theta}}
\newcommand{\C}{\mathbb{C}}
\newcommand{\Z}{\mathbb{Z}}
\newcommand{\R}{\mathbb{R}}
\newcommand{\ide}{\mathfrak{I}}
\newcommand{\e}{\varepsilon}
\newcommand{\ct}{C^\infty(T)}
\newcommand{\ore}{\mathfrak{A}}
\newcommand{\toep}{\mathcal{T}^\infty}
\newcommand{\comp}{\mathbb{K}}

\DeclareMathOperator{\coker}{coker}

\newtheorem{definition}{Definition}[section]
\newtheorem{proposition}[definition]{Proposition}
\newtheorem{lemma}[definition]{Lemma}
\newtheorem{theorem}[definition]{Theorem}
\newtheorem{ex}[definition]{Example}
\newtheorem{corollary}[definition]{Corollary}

\title{The Entire Cyclic Cohomology of Noncommutative 3-spheres}
{\rm
\author{{\rm{Katsutoshi Kawashima}} \and {\rm{Hiroshi Takai}}}
}

\begin{document}

\maketitle

\begin{center}
{\sc
Department of Mathematics and Information Sciences, \\
Tokyo Metropolitan University
}
\end{center}

\begin{abstract}
In this paper, we compute the entire cyclic cohomology of noncommutative 3-spheres. First of all, we verify the Mayer-Vietoris exact sequence of entire cyclic cohomology in the framework of Fr\'echet $^*$-algebras. Applying it to their noncommutative Heegaard decomposition, we deduce that their entire cyclic cohomology is isomorphic to the d'Rham homology of the ordinary 3-sphere with the complex coefficients.
\end{abstract}

\section{Introduction}

Since Connes \cite{connes} constructed a generalization of periodic cyclic cohomology which is called entire cyclic cohomology, its explicit computation is executed only for few examples (cf. \cite{brodzki-plymen,connes,mathai-stevenson}). As a matter of fact, their entire cyclic cohomologies are nothing but their periodic ones. Recently, the first named author \cite{naito} computed that of smooth noncommutative 2-tori, which have the same property cited above.

In this paper, we firstly formulate the Mayer-Vietoris exact sequence for entire cyclic cohomology, then we apply it to compute for smooth noncommutative 3-spheres. The key idea is based on Meyer's excision \cite{meyer,meyer2} concerning the short exact sequences of Fr\'echet $^*$-algebras to obtain a noncommutative Mayer-Vietoris exact sequence for entire cyclic cohomology. To use his excision, we need to construct a bounded linear section for a short exact sequence of Fr\'echet $^*$-algebras. To ensure it, we reformulate the notion of metric approximation property in the framework of Fr\'echet $^*$-algebras to solve the lifting problem (see \cite{choi}). We then use Baum, Hajac, Matthes and Szyma\'nskis' method \cite{bhms} for a Heegaard decomposition of smooth noncommutative 3-spheres since they pointed out an insufficient part of Matsumoto's costruction \cite{matsumoto} in the case of $C^*$-algebras. 

Under this circumstance, we conclude that the entire cyclic cohomology of noncommutative 3-spheres is the same as their periodic one.

Throughout this paper, $\theta$ is an irrational number in the open unit interval $(0, 1)$ and we use the notation $\Z_{\geq 0}$ for the set of all nonnegative integers.

\newpage

%%%%%%%%%%%%%%%%%%%%%%%%%%%%%%%%%%%%%%%%%%%%%%%%%
%%%%%%%%%%%%%% section 2 %%%%%%%%%%%%%%%%%%%%%%%%
%%%%%%%%%%%%%%%%%%%%%%%%%%%%%%%%%%%%%%%%%%%%%%%%%

\section{Preliminaries}

We prepare some notations and basic properties used throughout the paper. Let $\mathfrak{A}$ be a Fr\'echet $^*$-algebra or $F^*$-algebra and denote by $C^\infty([0, 1], \mathfrak{A})$ the set of all $\mathfrak{A}$-valued smooth functions on the closed unit interval $[0, 1]$ with respect to Fr\'echet topology. Given an element $f\in C^\infty([0, 1], \mathfrak{A})$ and an integer $n\geq 1$, we write by $f^{(n)}$ its $n$-th derivative of $f$ at $t \, (0<t<1)$ and denote by $f^{(n)}_+(0), f^{(n)}_-(1)$ the $n$-th derivatives at $0$ or $1$ as follows:
\begin{align*}
f_+^{(n)}(0)&=\lim_{t\to 0+}f^{(n)}(t) \\
f_-^{(n)}(1)&=\lim_{t\to 1-0}f^{(n)}(t).
\end{align*}
For $n=0$, we write $f^{(0)}_+(0)=f(0), f^{(0)}_-(1)=f(1)$.

\begin{definition}
For a $F^*$-algebra $\ore$, we define the suspension $S^\infty\ore$ of $\ore$ by
\[
S^\infty\mathfrak{A}=\{f\in C^\infty([0, 1], \mathfrak{A}) \, | \, f_+^{(n)}(0)=f_-^{(n)}(1)=0 \quad (n\geq 0) \}. 
\]
and we also define the cone $C^\infty\ore$ of $\ore$ by
\[
C^\infty\mathfrak{A}=\{f\in C^\infty([0, 1], \mathfrak{A}) \, | \, f_-^{(n)}(1)=0 \quad (n\geq 0) \}.
\]
\end{definition}

Then we have the following short exact sequence:
\[
\begin{CD}
0 @>>> \mathfrak{I} @>i>> C^\infty\mathfrak{A} @>q>> \mathfrak{A} @>>> 0,
\end{CD}
\]
where $q$ is defined by $q(f)=f(0)$, 
\[ \mathfrak{I}=\{ f\in C^\infty\mathfrak{A} \, | \, f(0)=0 \} \]
and $i$ is the canonical inclusion. The map $s : \ore \to C^\infty\ore$ defined by 
\[ s(a)(t)=(1-t)a \quad (a\in \mathfrak{A}, t\in [0, 1]) \]
is a bounded linear section of $q$ with respect to Fr\'echet topology. We need to know the entire cyclic cohomologies of $C^\infty\mathfrak{A}$ and $\mathfrak{I}$. We say that given two $F^*$-algebras $\mathfrak{A}$ and $\mathfrak{B}$, the map
\[ \Phi : \mathfrak{A} \to C^\infty ([0, 1], \mathfrak{B}) \]
is called a smooth homotopy if it is a bounded homomorphism with respect to Fr\'echet topology and two bounded homomorphisms $f, g : \mathfrak{A} \to \mathfrak{B}$ are smoothly homotopic if there exists a smooth homotopy $\Phi$ from $\mathfrak{A}$ to $\mathfrak{B}$ with $\Phi_0=f, \Phi_1=g$. A Fr\'echet algebra $\mathfrak{A}$ is smoothly homotopic to another one $\mathfrak{B}$ if there are two homomorphisms $f : \mathfrak{A} \to \mathfrak{B}$ and $g : \mathfrak{B} \to \mathfrak{A}$ such that $g\circ f$ (resp. $f\circ g$) is smoothly homotopic to the identity on $\mathfrak{A}$ (resp. $\mathfrak{B}$). According to Meyer \cite{meyer}, we know the homotopy invariance of entire cyclic cohomology in the framework of $F^*$-algebras:

\begin{proposition}[\cite{meyer}]\label{invariance}
If two bounded homomorphisms are smoothly homotopic, then they induce the same map on the entire cyclic cohmology.
\end{proposition}

We also mention the following lemma:

\begin{lemma}\label{iso_for_suspension}
Let $S^\infty\mathfrak{A}, C^\infty\mathfrak{A}$ and $\mathfrak{I}$ be cited above, we then have that
\[ HE^*(C^\infty\mathfrak{A})=0, \quad HE^*(\mathfrak{I})\simeq HE^*(S^\infty\mathfrak{A}). \]
\end{lemma}

\begin{proof}
By Proposition \ref{invariance}, it suffices to show that $C^\infty\mathfrak{A}$ is smoothly homotopic to $0$ to obtain the former isomorphism. The map
\[ F : C^\infty\mathfrak{A} \to C^\infty ([0, 1], C^\infty\mathfrak{A}) \]
defined by
\[ F_s(f)(t)=f(s+(1-s)t) \quad (f\in C^\infty\mathfrak{A}, \quad s, t\in [0, 1]) \]
gives a smooth homotopy on $C^\infty\mathfrak{A}$. Since $F_0$ is the identity on $C^\infty\mathfrak{A}$ and for any $f\in C^\infty\mathfrak{A}$,
\[ F_1(f)(t)=f(1)=0. \]

We know that $C^\infty\mathfrak{A}$ is smoothly homotopic to $0$. For the latter one, we introduce the map $t\mapsto f(e^{1-1/t}) \quad (f\in C^\infty\mathfrak{A}, \, t\in [0, 1])$, which belongs to $S^\infty\mathfrak{A}$. Indeed, we note that for any $n\geq 1$, $\dfrac{d^n}{dt^n}f(e^{1-1/t})$ is a linear combination of some functions such as
\[ f^{(k)}(e^{1-1/t})\frac{e^{l(1-1/t)}}{t^m} \quad (k, l, m\geq 1). \]
In fact, for $n=1$, we have that
\[ \frac{d}{dt}f(e^{1-1/t})=f^{(1)}(e^{1-1/t})\frac{e^{1-1/t}}{t^2}. \]
Suppose that the function $\dfrac{d^n}{dt^n}f(e^{1-1/t})$ is a linear combination of fuctions
\[ f^{(k)}(e^{1-1/t})\dfrac{e^{l(1-1/t)}}{t^m} \quad (k, l, m\geq 1), \]
then we deduce that
\begin{align*}
&\frac{d}{dt}\left(f^{(k)}(e^{1-1/t})\frac{e^{l(1-1/t)}}{t^m}\right) \\
&=f^{(k+1)}(e^{1-1/t})\dfrac{e^{(l+1)(1-1/t)}}{t^{m+2}}+lf^{(k)}\dfrac{e^{l(1-1/t)}}{t^{m+2}}-mf^{(k)}(e^{1-1/t})\dfrac{e^{l(1-1/t)}}{t^{m+1}},
\end{align*}
so is $\dfrac{d^{n+1}}{dt^{n+1}}f(e^{1-1/t})$. Because of the following equalities:
\begin{align*}
\lim_{t\to 0+}f^{(k)}(e^{1-1/t})\frac{e^{l(1-1/t)}}{t^m}&=f_+^{(n)}(0)\cdot 0=0 \\
\lim_{t\to 1-0}f^{(k)}(e^{1-1/t})\frac{e^{l(1-1/t)}}{t^m}&=f_-^{(n)}(1)=0,
\end{align*}
for any $f\in \mathfrak{I}, \, k, l, m\geq 1$, the function $f(e^{1-1/t})$ belongs to $S^\infty\mathfrak{A}$. Let
\[ r : \mathfrak{I} \to S^\infty\mathfrak{A} \]
be the map defined by
\[ r(f)(t)=f(e^{1-1/t}) \quad (f\in \mathfrak{I}, \, t\in [0, 1]) \]
and $i$ the natural inclusion from $S^\infty\mathfrak{A}$ into $\mathfrak{I}$. For the proof that $r\circ i$ is smoothly homotopic to the identity on $S^\infty\mathfrak{A}$, we use the bounded homomorphism
\[ G : S^\infty\mathfrak{A} \to C^\infty([0, 1], S^\infty\mathfrak{A}) \]
defined by
\[ G_s(f)(t)=f(se^{1-1/t}+(1-s)t) \quad (f\in S^\infty\mathfrak{A}, \, s, t\in [0, 1]) \]
which gives a smooth homotopy connecting $r\circ i$ and the identity on $S^\infty\mathfrak{A}$. We firstly show that $G_s(f)\in S^\infty\mathfrak{A}$ for any fixed $f\in S^\infty\mathfrak{A}, s\in [0, 1]$. Since
\[ \dfrac{d}{dt}G_s(f)(t)=f^{(1)}(se^{1-1/t}+(1-s)t)\left(\dfrac{s}{t^2}e^{1-1/t}+1-s\right), \]
we know that
\begin{align*}
\lim_{t\to 0+}\dfrac{d}{dt}G_s(f)(t)&=f_+^{(1)}(0)\cdot (1-s)=0 \\
\lim_{t\to 1-0}\dfrac{d}{dt}G_s(f)(t)&=f_-^{(1)}(1)=0. \\
\end{align*}
For general $n\geq 2$, we also see that
\[ \lim_{t\to 0+}\dfrac{d^n}{dt^n}G_s(f)(t)=\lim_{t\to 1-0}\dfrac{d^n}{dt^n}G_s(f)(t)=0. \]
The case for $n=1$ has already been shown. It suffices to show that for $n\geq 2$, the fuction $\dfrac{d^n}{dt^n}G_s(f)(t)$ is a linear combination of fuctions like
\[ f^{(k)}(se^{1-1/t}+(1-s)t)\dfrac{e^{l(1-1/t)}}{t^m} \quad (k, l, m\geq 1). \]
We now calculate that
\begin{align*}
&\dfrac{d}{dt}f^{(k)}(se^{1-1/t}+(1-s)t)\dfrac{e^{l(1-1/t)}}{t^m} \\
&=f^{(k+1)}(se^{1-1/t}+(1-s)t)\dfrac{e^{l(1-1/t)}}{t^m}\left(\dfrac{s}{t^2}+1-s\right) \\
&\quad +f^{(k)}(se^{1-1/t}+(1-s)t)\left(\dfrac{le^{l(1-1/t)}}{t^{m+2}}-\dfrac{me^{l(1-1/t)}}{t^{m+1}}\right),
\end{align*}
which completes the induction process. Moreover we see that $\dfrac{d^n}{dt^n}G_s$ is uniformly bounded on $[0, 1]$ for each $n\geq 1$. We note that the function
\[ t\mapsto \dfrac{e^{l(1-1/t)}}{t^m} \]
is bounded on $[0, 1]$ and that $f^{(k)}$ is also bounded since $f\in C^\infty([0, 1], \mathfrak{A})$. Hence $G$ is a smooth homotopy connecting $r\circ i$ and the identity on $S^\infty\mathfrak{A}$ since $G_1=r\circ i$ and $G_0$ is the identity on $S^\infty\mathfrak{A}$. Similarly, $i\circ r$ and the identity on $\mathfrak{I}$ are smoothly homotopic via the smooth homotopy defined by the same way as $G$, which implies that
\[ HE^*(\mathfrak{I})\simeq HE^*(S^\infty\mathfrak{A}) \]
as desired.
\end{proof}

%%%%%%%%%%%%%%%%%%%%%%%%%%%%%%%%%%%%%%%%%%%%%%%%%%%%
%%%%%%%%%%%%%%% section 3 %%%%%%%%%%%%%%%%%%%%%%%%%%
%%%%%%%%%%%%%%%%%%%%%%%%%%%%%%%%%%%%%%%%%%%%%%%%%%%%

\section{Toeplitz $F^*$-Algebras}
In this section, we construct smooth Toeplitz algebras based on 1-torus and to analyze them. They could be viewd as a quantization of 2-disc (cf. \cite{bhms,klimek}). Let $\{z^n\}_{n\in\Z}$ be the orthonomal basis of the Hilbert space $L^2(T)$ of all square integrable functions on the 1-torus $T$, where $z^n(t)=t^n \, (t\in T, n\in\Z)$, and $H^2=H^2(T)$ the Hardy space on $T$ which is a closed subspace of $L^2(T)$ spanned by $\{z^n\}_{n\geq 0}$. For $f\in C^\infty(T)$ of all infinitely differentiable functions on $T$, in which we mean that the derivation is defined by
\[
\frac{d}{dt}f(t)=\lim_{r\to 0}\frac{f(e^{2\pi ir}t)-f(t)}{r},
\]
we define the operator $T_f$ for $f\in\ct$ by
\[ T_f\xi=Pf\xi \quad (\xi\in H^2), \]
where $P$ is the projection onto $H^2$. We consider the $^*$-algebra $\mathcal{P}$ generated by $T_{z^j}\, (j\in\Z)$, namely,
\[ \mathcal{P}=\bigcup_{N\in\Z_{\geq 0}}\left\{ \left. \sum_{i_j\in \Z, \, |i_j|\leq N}c_{i_1, \dots , i_n}T_{z^{i_1}}\dots T_{z^{i_n}} \right| \, c_{i_1, \dots , i_n}\in\C, n\in\Z_{\geq 0} \right\}. \]
Since $T_fT_g-T_{fg}$ is a compact operator for any $f, g\in\ct$ and $T_f$ is compact if and only if $f=0$ (cf. \cite{davidson}), it is easily seen by induction that for any $T\in\mathcal{P}$, there is a unique $f\in\ct$ and a unique compact operator $S$ with $T=T_f+S$. Actually, if
\[ T=\sum_{i_j\in \Z, \, |i_j|\leq N}c_{i_1, \dots , i_n}T_{z^{i_1}}\dots T_{z^{i_n}}\in\mathcal{P}, \]
then $T=T_f+S$, where
\[ f=\sum_{i_j\in \Z, \, |i_j|\leq N}c_{i_1, \dots , i_n}z^{i_1}\dots z^{i_n} 
\]
and the compact operator $S$ is a linear combination of the operators of the form
\[
T_{z^{l_1}}\cdots T_{z^{l_k}}(T_{z^n}T_{z^m}-T_{z^{n+m}})T_{z^{l'_1}}\cdots T_{z^{l'_{k'}}} \quad (l_1, \dots l_{k}, \dots l'_1, \dots l'_{k'}, n, m\in\Z ).
\]
We show that there exists a function $K_S(t, s)\in C^\infty(T^2)$ which is a polynomial of $t, s$ and satisfies
\[ (S\xi)(t)=\int_T K_S(t, s)\xi(s)ds. \quad (\xi\in H^2). \]
This function $K_S$ is called the kernel function of $S$. Given $n, m\in\Z$, it is easily verified that
\begin{align*}
(T_{z^n}T_{z^m}-T_{z^{n+m}})\xi(t)&=\left(\sum_{k\geq \max\{-m, -n\}}-\sum_{k\geq -m-n}\right)\langle\xi\, |\, z^k\rangle z^k(t) \\
&=\int_T\left(\sum_{k\geq \max\{-m, -n\}}-\sum_{k\geq -m-n}\right)t^ks^{-k}\xi(s)ds,
\end{align*}
where 
\[ \langle f \, | \, g\rangle=\int_T f(s)\overline{g(s)}ds \quad (f, g\in L^2(T)) \]
is the usual inner product on $L^2(T)$. Then the kernel function $K_{T_{z^n}T_{z^m}-T_{z^{n+m}}}$ of $T_{z^n}T_{z^m}-T_{z^{n+m}}$ is a finitie sum of the fuctions $t^ks^{-k}$ since there exists a finite subset $I_{n, m}\subset\Z$ such that
\[
K_{T_{z^n}T_{z^m}-T_{z^{n+m}}}(t, s)=\left(\sum_{k\geq \max\{-m, -n\}}-\sum_{k\geq -m-n}\right)t^ks^{-k}=\pm\sum_{k\in I_{n, m}}t^ks^{-k}
\]
(when $I_{n, m}$ is empty, we regard the function $K_{T_{z^n}T_{z^m}-T_{z^{n+m}}}(t, s)=0$). Moreover, given $l\in\Z$, we compute that
\begin{align*}
K_{(T_{z^n}T_{z^m}-T_{z^{n+m}})T_{z^l}}(t, s)&=\int_T K_{T_{z^n}T_{z^m}-T_{z^{n+m}}}(t, r)K_{T_{z^l}}(r, s)dr \\
&=\pm\int_T\sum_{k\in I_{n, m}} t^k r^{-k}\sum_{k'\geq -l}r^{k'}s^{-k'}dr \\
&=\pm\sum_{k'\geq -l}\sum_{k\in I_{n, m}}t^ks^{-k'}\int_Tr^{k'-k}dr \\
&=\pm\sum_{k\in I_{n, m}, \, k\geq -l}t^ks^{-k},
\end{align*}
which implies that the kernel function $K_{(T_{z^n}T_{z^m}-T_{z^{n+m}})T_{z^l}}$ is a polynomial of $t, s$. By the similar computation, it follows that for $l, m, n\in\Z$, the kernel function \\
$K_{T_{z^l}(T_{z^n}T_{z^m}-T_{z^{n+m}})}$ is also a polynomial. Then, by the inductive argument, we have that the kernel functions
\[
K_{T_{z^{l_1}}\cdots T_{z^{l_k}}(T_{z^n}T_{z^m}-T_{z^{n+m}})T_{z^{l'_1}}\cdots T_{z^{l'_{k'}}}} \]
are also polynomials, which in particular belong to $C^\infty(T^2)$.

Let $\comp^\infty$ be the set of all compact operators $S$ such that there exists a function $K_S\in C^\infty(T^2)$ with the property that
\[ (S\xi)(t)=\int_T K_S(t, s)\xi(s)ds \quad (\xi\in H^2, \, t\in T). \]
By the above argument, it follows that for each operator $T\in\mathcal{P}$, there exist a function $f\in\ct$ and an operator $S\in\comp^\infty$ with $T=T_f+S$. Since $T_g$ is compact if and only if $g=0$, the function $f$ and the operator $S$ are uniquely determined. We define the seminorms $\{\|\cdot\|_{k, l, m}\}$ on $\mathcal{P}$ by
\[ \| T_f+S \|_{k, l, m}=\| f^{(k)} \|_\infty+\|K_S^{(l, m)}\|_\infty \quad (k, l, m\in\Z_{\geq 0}), \]
where $f^{(k)}$ is the $k$-th derivative of $f$,
\[ K^{(l, m)}=\dfrac{\partial^{l+m}}{\partial t^l \partial s^m}K(t, s) \quad (K\in C^\infty(T^2)), \]
and $\|\cdot\|_\infty$ mean the supremum norms on the corresponding function spaces. 
\begin{definition}
The smooth Toeplitz algebra $\toep$ is defined by the completion of $\mathcal{P}$ with respect to the topology induced by the seminorms $\{\|\cdot\|_{k, l, m}\}$.
\end{definition}
Similarly as in the case of $\mathcal{P}$, we have that for any $T\in\toep$, there exist a function $f\in\ct$ and an operator $S\in\comp^\infty$ with $T=T_f+S$. In fact, if $\{T_n\}_{n\geq 1}\subset\mathcal{P}$ converges to $T$ with respect to the seminorms $\{\|\cdot\|_{k, l, m}\}$ with $T_n=T_{f_n}+S_n$, we compute that
\begin{align*}
\|f_n^{(k)}-f_{n'}^{(k)}\|_\infty&=\|T_{f_n}-T_{f_{n'}}\|_{k, 0, 0} \\
&\leq \|T_n-T_{n'}\|_{k, 0, 0}\to 0 \quad (n, n'\to\infty),
\end{align*}
for any $k\in\Z_{\geq 0}$, which ensures that there exists the function $f\in\ct$ such that $f_n\to f$ with respect to the seminorms. Alternatively, since $\{ S_n \}$ is also Cauchy, we have that for any $k, l, m\in\Z$,
\[ \|S_n-S_{n'}\|_{k, l, m}=\|K_{S_n}^{(l, m)}-K_{S_{n'}}^{(l, m)}\|_\infty \to 0 \quad (n, n'\to\infty). \]
Hence, we find a function $K\in C^\infty(T^2)$ with $K_{S_n}\to K$ as $n\to\infty$ with respect to Fr\'echet topology on $C^\infty(T^2)$. Then the operator $S$ defined by
\[ S\xi(t)=\int_T K(t, s)\xi(s)ds \quad (\xi\in H^2, \, t\in T) \]
belongs to $\comp^\infty$ and $S_n-S\to 0$ as $n\to\infty$ with respect to the seminorms, which implies the conclusion. It is clear by the above argument that $\comp^\infty$ is a $^*$-ideal of $\toep$ and Fr\'echet closed. 

We define a homomorphism $q : \toep \to \ct$ by $q(T_f+S)=f$, which is continuous with respect to the seminorms cited before. The following lemma is already clear:

\begin{lemma}
We obtain the following short exact sequence as $F^*$-algebras:
\[
\begin{CD}
0 @>>> \comp^\infty @>i>> \toep @>q>> \ct @>>> 0,
\end{CD}
\]
where $i$ is the canonical inclusion.
\end{lemma}
We next deduce the following lemma, which is a smooth version of $C^*$-algebra case:

\begin{lemma}\label{compact}
We have the following isomorphism:
\[ \comp^\infty\simeq \varinjlim (M_n(\C), \varphi_{n}), \]
where the homomorphisms $\varphi_{n} : M_n(\C)\to M_{n+1}(\C)$ are given by
\[ \varphi_n(A)=
\begin{pmatrix}
A & 0 \\
0 & 0
\end{pmatrix}
\quad (A\in M_{n}(\C), \, n\geq 1).
\]
\end{lemma}

\begin{proof}
Let $P_n \, (n\geq 1)$ be the orthogonal projections on $H^2$ defined by
\begin{align*}
P_n\xi(t)&=\sum_{k=0}^{n-1}\langle \xi \, | \, z^k\rangle z^k(t) \\
&=\sum_{k=0}^{n-1}\left(\int_T\xi(s)s^{-k}ds\right)t^k \\
&=\int_T\sum_{k=0}^{n-1}t^ks^{-k}\xi(s)ds \quad (\xi\in H^2),
\end{align*}
which implies that
\[
K_{P_n}(t, s)=\sum_{k=0}^{n-1} t^ks^{-k}. 
\]
Then $P_n\comp^\infty P_n$ is isomorphic to $M_n(\C)$. Indeed, the kernel function $K_{P_nSP_n}$ for $S\in\comp^\infty$ is calculated as follows: since
\begin{align*}
K_{SP_n}(t, s)&=\int_T K_S(t, r)K_{P_n}(r, s)dr \\
&=\int_T\sum_{k=0}^nr^ks^{-k}K_S(t, r)dr,
\end{align*}
we have that
\begin{align*}
K_{P_nSP_n}(t, s)&=\int_TK_{P_n}(t, u)K_{SP_n}(u, s)du \\
&=\int_T\sum_{k=0}^{n-1}t^ku^{-k}\left(\int_T\sum_{k'=0}^{n-1}r^{k'}s^{-k'}K_S(u, r)dr\right)du \\
&=\int_T\int_T\sum_{k, k'=0}^{n-1}t^ku^{-k}r^{k'}s^{-k'}K_S(u, r)drdu \\
&=\sum_{k, k'=0}^{n-1}t^ks^{-k'}\int_T\int_Tr^{k'}u^{-k}K_S(u, r)drdu \\
&=\sum_{k, k'=0}^{n-1}c_{k, k'}t^ks^{-k'},
\end{align*}
where
\[ c_{k, k'}=\int_T\int_Tr^{k'}u^{-k}K_S(u, r)drdu \]
are the Fourier coefficients of $K_S\in C^\infty(T^2)$.

On the other hand, we define the matrix units $E_{ij}$ in what follows: when $i=j$, we define
\[
E_{ii}=T_{z^{i-1}}T_{z^{i-1}}^*-T_{z^i}T_{z^i}^*.
\]
For $i\not=j$, we define
\[
E_{ij}=
\begin{cases}
T_{z^{j-i}}E_{ii} & (i<j) \\
E_{jj}T_{z^{j-i}} & (i>j).
\end{cases}
\]
It is not hard to see that $\{E_{ij}\}$ forms a family of matrix units. By taking $m=-n$ in the computation of the kernel function of $T_{z^n}T_{z^m}-T_{z^{n+m}}$, we have
\[
K_{I-T_{z^n}T_{z^n}^*}(t, s)=\sum_{k=0}^{n-1}t^ks^{-k}.
\]
Hence we have
\[
K_{E_{ii}}(t, s)=t^{i-1}s^{-(i-1)}.
\]
More generally, we obtain that
\[
K_{E_{ij}}(t, s)=t^{j-1}s^{-(i-1)}.
\]
Then $P_n\comp^\infty P_n$ is generated by the matrix units $\{E_{ij}\}_{i, j=1}^n$ so that it is isomorphic to $M_n(\C)$ with the seminorms given by
\[
\|(\lambda_{kl})\|_{p, q}=\sup_{t, s\in T}\left|\sum_{k, l=0}^{n-1}\lambda_{kl}l^pk^qt^ls^{-k}\right| \quad ((\lambda_{kl})\in M_n(\C)).
\]
For any $S\in\comp^\infty$, $\|S-P_nSP_n\|_{l, m}\to 0$ as $n\to\infty$ for any $l, m\geq 0$ since $\{c_{k, k'}\}$ belongs to the Schwartz space on $\Z^2$. Therefore,
\[ \|S-P_nSP_n\|_{l, m}\to 0 \quad (n\to\infty) \]
for any $l, m\in\Z_{\geq 0}$. Hence, the conclusion follows.
\end{proof}

By the above lemma, we deduce the following corollaries:

\begin{corollary}
$\comp^\infty$ is a simple $F^*$-algebra, which is equal to the commutator $F^*$-ideal $[\toep, \toep]$ of $\toep$.
\end{corollary}

In what follows, we study briefly the $F^*$-crossed products $\toep\rtimes_{\alpha_\theta}\Z$ of $\toep$ by the gauge action $\alpha_\theta$ of $\Z$. Let $\alpha_\theta$ be the action of $\Z$ on $\toep$ defined by
\[ \alpha_\theta(T_f)=T_{f_{\theta}} \quad (f\in\ct, n\in\Z), \]
where $f_{\theta}(z)=f(e^{2\pi i\theta}z)$, which gives a $F^*$-dynamical system $(\toep, \Z, \alpha_\theta)$. We also consider the unitary operator $U_\theta$ on $H^2$ defined by
\[
U_\theta\xi(t)=\xi(e^{2\pi i\theta}t) \quad (\xi\in H^2, t\in T).
\]
It is easily seen that $U_\theta^*\xi(t)=U_\theta^{-1}\xi(t)=\xi(e^{-2\pi i\theta}t)$. Then we form the $F^*$-crossed products $\toep\rtimes_{\alpha_\theta}\Z$ of the $F^*$-dynamical system $(\toep, \Z, \alpha_\theta)$, which could be viewd as the deformation quantization $(D^2\times S^1)_\theta$ of the solid torus $D^2\times S^1$. In fact, let $\toep[\Z]$ be the $^*$-algebra of all finite sums
\[
f=\sum_{n\in\Z, \, |n|\leq N}A_nU_\theta^n \quad (A_n\in\toep, \, N\in\Z_{\geq 0}),
\]
where its multiplication is determined by $U_\theta AU_\theta^{-1}=\alpha_\theta(A)$ and its $^*$-operation is given by $(AU_\theta)^*=\alpha_\theta^{-1}(A^*)U_\theta^{-1}$. For $f=\sum A_nU_\theta^n\in\toep[\Z]$, we induce the seminorms defined by
\[
\|f\|_{p, q, r, s}=\sup_{n\in\Z}(1+|n|^2)^p\|A_n\|_{q, r, s} \quad (p, q, r, s\in\Z_{\geq 0}).
\]
We define the $F^*$-crossed product $(D^2\times S^1)_\theta=\toep\rtimes_{\alpha_\theta}\Z$ by the completion of $\toep[\Z]$ with respect to the seminorms cited above. For $S\in\comp^\infty$, we calculate that
\begin{align*}
\alpha_\theta(S)\xi(t)&=U_\theta SU_\theta^* \xi(t) \\
&=U_\theta\int_T K_S(t, s)\xi(e^{-2\pi i\theta}s)ds \\
&=\int_T K_S(e^{2\pi i\theta}t, e^{2\pi i\theta}s)\xi(s)ds
\end{align*}
to obtain that
\[
K_{\alpha_\theta(S)}(t, s)=K_S(e^{2\pi i\theta}t, e^{2\pi i\theta}s).
\]
Therefore, we have $\alpha_\theta(\comp^\infty)=\comp^\infty$ so that we construct a $F^*$-dynamical system $(\comp^\infty, \Z, \alpha_\theta)$. Since
\begin{align*}
K_{\alpha_\theta(P_n)}(t, s)&=K_{P_n}(e^{2\pi i\theta}t, e^{2\pi i\theta}s) \\
&=\sum_{k=0}^{n-1}(e^{2\pi i\theta}t)^k(e^{2\pi i\theta}s)^{-k} \\
&=\sum_{k=0}^{n-1}t^ks^{-k}=K_{P_n}(t, s),
\end{align*}
we have $\alpha_\theta(P_n\comp^\infty P_n)=P_n\comp^\infty P_n$. Therefore, we also construct $F^*$-dynamical systems $(P_n\comp^\infty P_n, \Z, \alpha_\theta^{(n)})$, where $\alpha_\theta^{(n)}$ are the restrictions of $\alpha_\theta$ on $P_n\comp^\infty P_n$. Let $i_n$ be the isomorphism from $P_n\comp^\infty P_n$ onto $M_n(\C)$ defined before and $i=\varinjlim i_n$ the isomorphism from $\varinjlim P_n\comp^\infty P_n$ onto $\comp^\infty$ induced by the isomorphisms $i_n$. We write by $\overline{\alpha}_\theta^{(n)}$ the action $i_n\circ\alpha_\theta\circ i_n^{-1}$.

\begin{proposition}\label{iso:1}
We have the following isomorphism:
\[
\comp^\infty\rtimes_{\alpha_\theta}\Z\simeq \varinjlim (M_n(\C)\rtimes_{\overline{\alpha}_\theta^{(n)}}\Z, \widetilde{\varphi}_n),
\]
where $\widetilde{\varphi}_n$ are the inclusions induced naturally by $\varphi_n$.
\end{proposition}

\begin{proof}
Since $i\circ\overline{\alpha}_\theta^{(n)}=\alpha_\theta\circ i$ for any $n\geq 1$, we have
\[
\comp^\infty\rtimes_{\alpha_\theta}\Z\simeq \varinjlim (P_n\comp^\infty P_n\rtimes_{\alpha_\theta^{(n)}}\Z, \varphi_n),
\]
where $\varphi_n : P_n\comp^\infty P_n\rtimes_{\alpha_\theta^{(n)}}\Z \to P_{n+1}\comp^\infty P_{n+1}\rtimes_{\alpha_\theta^{(n+1)}}\Z$ are the canonical inclusions. Moreover, since $i_n\circ\alpha_\theta^{(n)}=\overline{\alpha}_\theta^{(n)}\circ i_n$, we find isomorphisms
\[
\psi_n : P_n\comp^\infty P_n\rtimes_{\alpha_\theta^{(n)}}\Z \overset{\simeq}{\longrightarrow} M_n(\C)\rtimes_{\overline{\alpha}_\theta^{(n)}}\Z.
\]
Then since $\psi_n\circ\varphi_n=\widetilde{\varphi}_n\circ\psi_n$, we conclude
\[
\varinjlim (P_n\comp^\infty P_n\rtimes_{\alpha_\theta^{(n)}}\Z , {\varphi}_n)
\simeq \varinjlim (M_n(\C)\rtimes_{\overline{\alpha}_\theta^{(n)}}\Z, \widetilde{\varphi}_n)
\]
as desired.
\end{proof}

Then we construct a $^*$-homomorphism 
\[
\rho_n : M_n(\C)\rtimes_{\overline{\alpha}_\theta^{(n)}}\Z \to M_n(\C)\hat{\otimes}_\gamma \mathcal{S}(\Z),
\]
where $\mathcal{S}(\Z)$ is the set of all rapidly decreasing sequences $\{c_n\}\subset \C$ and $\hat{\otimes}_\gamma$ means the tensor product of $F^*$-algebras completed by the topology induced by the seminorms defined by
\[
\left\|\sum_{j=1}^N x_j\otimes y_j\right\|_{k, l}=\inf\sum_{j=1}^N\|x_j\|_k\|y_j\|_l,
\]
where the infinimum is taken over the all representations of $\sum_{j=1}^N x_j\otimes y_j$. Equivalently, $M_n(\C)\hat{\otimes}_\gamma S(\Z)$ is regarded as $\mathcal{S}(\Z, M_n(\C))$ with the ordinary convolution as its product. For $x\in M_n(\C)\rtimes_{\alpha_\theta^{(n)}}\Z$, we define
\[
\rho_n(x)=xU_\theta^n.
\]
It is easily seen that it is an isomorphism. Moreover, since
\begin{align*}
\|\rho_n(x)\|_{p, q, r, s}=\sup_{m\in\Z}(1+m^2)^p\|x_mU_\theta^n\|_{q, r, s}=\|x\|_{p, q, r, s} \quad (p, q, r, s\in\Z_{\geq 0}),
\end{align*}
for any $x=\sum x_mU_\theta^m\in\toep\Z$, it is Fr\'echet isometry. Therefore, we have
\[
M_n(\C)\rtimes_{\overline{\alpha}_\theta^{(n)}}\Z\simeq M_n(\C)\hat{\otimes}_\gamma \mathcal{S}(\Z)
\]
by $\rho_n$. Now it is immediately known that the following fact follows:

\begin{corollary}\label{stableconti}
The isomorphism
\[
\comp^\infty\rtimes_{\alpha_\theta}\Z\simeq \comp^\infty\hat{\otimes}_\gamma \ct
\]
holds.
\end{corollary}

\begin{proof}
By Proposition \ref{iso:1} and Lemma \ref{compact}, we have that
\begin{align*}
\comp^\infty\rtimes_{\alpha_\theta}\Z &\simeq \varinjlim (M_n(\C)\rtimes_{\overline{\alpha}_\theta^{(n)}}\Z, \widetilde{\varphi}_n) \\
&\simeq \varinjlim (M_n(\C)\hat{\otimes}_\gamma \mathcal{S}(\Z), \widetilde{\varphi}_n\otimes id_{\mathcal{S}(\Z)}) \\
&\simeq \left(\varinjlim (M_n(\C), \varphi_n)\right)\hat{\otimes}_\gamma \mathcal{S}(\Z) \\
&\simeq \comp^\infty \hat{\otimes}_\gamma \mathcal{S}(\Z).
\end{align*}
Since $\mathcal{S}(\Z)$ is isomorphic to $\ct$ sending by the Fourier transform, the conclusion follows.
\end{proof}

We end this section by stating the following fact:

\begin{corollary}\label{seq:2}
We have the following short exact sequence:
\[
\begin{CD}
0 @>>> \comp^\infty\otimes_\gamma \mathcal{S}(\Z) @>\widetilde{i}>> (D^2\times S^1)_\theta @>\widetilde{q}>> \ct\rtimes_{\overline{\alpha}_\theta}\Z @>>> 0,
\end{CD}
\]
where $\overline{\alpha}_\theta : \ct\times\Z \to \ct$ is the Fr\'echet continuous action defined by
\[
\overline{\alpha}_\theta^n(f)(z)=f(e^{2\pi in\theta}z) \quad (f\in\ct, z\in T),
\]
with a bounded linear section $\widetilde{s}$ of $\widetilde{q}$.
\end{corollary}

\begin{proof}
Since $i\circ \alpha_\theta^n=\alpha_\theta^n\circ i$ and $q\circ \alpha_\theta^n=\overline{\alpha}_\theta^n\circ q$ for all $n\in\Z$, it is clear that the desired short exact sequence holds and $\widetilde{s}(fU_\theta^n)=T_fU_\theta^n \, (f\in\ct, \, n\in\Z)$.
\end{proof}

%%%%%%%%%%%%%%%%%%%%%%%%%%%%%%%%%%%%%%%%%%%%%%%%%%%%%%%
%%%%%%%%%%%%%%%%%% section 4 %%%%%%%%%%%%%%%%%%%%%%%%%%
%%%%%%%%%%%%%%%%%%%%%%%%%%%%%%%%%%%%%%%%%%%%%%%%%%%%%%%

\section{Metric Approximation Property}

We introduce an analogue of the notion of metric approximation property for Banach spaces \cite{choi}. Let $\ore, \mathfrak{B}$ be two Banach spaces and $\mathfrak{I}\subset\mathfrak{B}$ an $M$-ideal. In \cite{choi}, the authors prove that if $\ore$ is separable and has the metric approximation property, then each contractive map $\varphi : \ore\to\mathfrak{B}/\mathfrak{I}$ has a lift $\widetilde{\varphi} : \ore \to \mathfrak{B}$ which is contractive and satisfies $q\circ \widetilde{\varphi}=\varphi$, where $q : \mathfrak{B}\to\mathfrak{B}/\mathfrak{I}$ is the quotient map. Our purpose in this section is to define this property for $F^*$-algebras to prove lifting problem cited above. The topology on $\ore$ induced by its seminorms $\{\|\cdot\|_k\}_{k\geq 0}$ is same as that induced by the metric $d_{\ore}$ defined by
\[ d_{\ore}(a, b)=\sum_{k=0}^{\infty}\dfrac{1}{\,\, 2^k}\dfrac{\|a-b\|_k}{1+\|a-b\|_k} \quad (a, b\in\ore). \]
We say that a linear map $\varphi : \ore \to \mathfrak{B}$ is bounded if and only if there exists a constant $C>0$ with
\[
d_{\mathfrak{B}}(\varphi(a), 0)\leq Cd_{\ore}(a, 0) \quad (a\in\ore).
\]

\begin{definition}
Let $\ore$ be a $F^*$-algebra and $\{\|\cdot\|_k\}_{k\geq 0}$ its seminorms. We say that it has the metric approximation property if there exists a family of bounded linear maps $\{\theta_n\}_{n\geq 1}$ on $\ore$ with the following properties:
\begin{enumerate}
\item each $\theta_n$ has a finite rank, 
\item for any $a\in\ore$, $d_{\ore}(\theta_n(a), a)\to 0$ as $n\to\infty$.
\end{enumerate}
\end{definition}

We give some examples of $F^*$-algebras with the metric approximation property. Here we note that $d_{\ore}(\theta_n(a), a)\to 0$ is satisfied if and only if $\|\theta_n(a)-a\|_k \to 0$ for any $k\geq 0$. 

\begin{ex}
{\rm
For an integer $n\geq 2$, let $\mathbb{F}_n$ be the free group with $n$-generators. Given $g\in\mathbb{F}_n$, we denote by $|g|$ its word length, and for $f\in\C [\mathbb{F}_n]$ and an integer $k\in\Z_{\geq 0}$, we define seminorms by
\[ \|f\|_k=\sup_{g\in\mathbb{F}_n}(1+|g|)^k|f(g)|. \]
The Schwartz space $\mathcal{S}(\mathbb{F}_n)$ is defined by the completion of $\C [\mathbb{F}_n]$ with respect to the above seminorms. For $f\in\mathcal{S}(\mathbb{F}_n)$, we define an bounded operator $\lambda(f)$ on the Hilbert space $l^2(\mathbb{F}_n)$ by the convolution with $f$, that is,
\[ (\lambda(f)\xi)(g)=(f*\xi)(g)=\sum_{h\in\mathbb{F}_n}f(h)\xi(h^{-1}g) \quad (g\in\mathbb{F}_n, \xi\in l^2(\mathbb{F}_n)), \]
on which the seminorms are defined by
\[ \|\lambda(f)\|_k=\|f\|_k \quad (k\in\Z_{\geq 0}). \]
This definition is well-defined. Indeed, if $\lambda(f)=0$, then $\lambda(f)\delta_e=0$, where $e\in\mathbb{F}_n$ is the unit and the element $\delta_e\in l^2(\mathbb{F}_n)$ is defined by
\[
\delta_e(g)=
\begin{cases}
1 \quad (g=e) \\
0 \quad (g\not=e).
\end{cases}
\]
Hence, for any $g\in\mathbb{F}_n$, we have that
\begin{align*}
0&=(\lambda(f)\delta_e)(g)=(f*\delta_e)(g) \\
&=\sum_{h\in\mathbb{F}_n}f(h)\delta_e(h^{-1}g)=f(g).
\end{align*}
Therefore, $\lambda(f)=0$ leads $f=0$, which implies that the seminorms are well-defined. We define the $F^*$-algebra $C^*_r(\mathbb{F}_n)^\infty$ by the completion of the $^*$-algebra generated by the bounded operators $\lambda(f)\, (f\in\mathcal{S}(\mathbb{F}_n))$. Here we claim that
\[ C^*_r(\mathbb{F}_n)^\infty=\{\lambda(f) \, | \, f\in\mathcal{S}(\mathbb{F}_n)\}. \]
In fact, it is clear that $\lambda(f)^*=\lambda(f^*)$ for any $f\in \mathcal{S}(\mathbb{F}_n)$, where 
\[ f^*(g)=\overline{f(g^{-1})}\in \mathcal{S}(\mathbb{F}_n). \]
For any $T\in C^*_r(\mathbb{F}_n)^\infty$, there exits a family $\{\lambda(f_n)\}_{n\geq 1} \, (f_n\in\mathcal{S}(\mathbb{F}_n))$ which converges to $T$ with respcet to the seminorms cited above. Then, for any $k\in\Z_{\geq 0}$, we have that
\[
\|f_n-f_m\|_k=\|\lambda(f_n)-\lambda(f_m)\|_k \to 0 \quad (n, m\to\infty), 
\]
which implies that there exists a function $f\in\mathcal{S}(\mathbb{F}_n)$ which is the limit of $\{f_n\}$. Thus, for any $k\in\Z_{\geq 0}$, we have that
\[ \|T-\lambda(f)\|_k \leq \|T-\lambda(f_n)\|_k+\|\lambda(f_n)-\lambda(f)\|_k \to 0 \quad (n\to\infty). \]
We construct a family of finite dimensional bounded linear maps $\{\theta_k\}$ on $C^*_r(\mathbb{F}_n)^\infty$ in what follows. Given an integer $k\geq 1$, let $E_k=\{ g\in\mathbb{F}_n \, | \, |g|\leq k\}$ and $\chi_k$ the function on $\mathcal{S}(\mathbb{F}_n)$ defined by
\[
\chi_k(g)=
\begin{cases}
1 \quad (g\in E_k) \\
0 \quad (g\not\in E_k).
\end{cases}
\]
Since the number of elements of $E_k$ is finite for each $k\geq 1$, the linear maps $\psi_n : \mathbb{F}_n \to \mathbb{F}_n$ defined by
\[ \psi_k(\lambda(f))=\lambda(\chi_k f) \]
have finite ranks. We define the finite linear bounded maps $\theta_k : \mathbb{F}_n \to \mathbb{F}_n \, (k\geq 1)$ by
\[ \theta_k(\lambda(f))=\lambda(e^{-|\cdot|/k}\chi_k f). \]
Then for any $l\geq 0$ and $f\in C^*_r(\mathbb{F}_n)^\infty$, we compute that
\begin{align*}
&\| \lambda(f)-\theta_k(\lambda(f))\|_l \\
&\leq \|\lambda(f)-\lambda(e^{-|\cdot |/k}f)\|_l+\|\lambda(e^{-|\cdot|/k}f)-\lambda(e^{-|\cdot|/k}\chi_kf)\|_l \\
&=\|f-e^{-|\cdot |/k}f\|_l+\|e^{-|\cdot|/k}f-e^{-|\cdot|/k}\chi_kf\|_l \\
&=\sup_{g\in\mathbb{F}_n}\left|(1+|g|)^lf(g)(1-e^{-|g|/k})\right|+\sup_{g\in\mathbb{F}_n}\left|(1+|g|)^lf(g)e^{-|g|/k}(1-\chi_k(g))\right| \\
&\leq\|f\|_l\sup_{g\in\mathbb{F}_n}\left|(1-e^{-|g|/k})\right|
+\sup_{|g|\geq k+1}\left|(1+|g|)^lf(g)e^{-|g|/k}\right| \\
&\to 0 \quad (k\to\infty).
\end{align*}
Therefore, $C^*_r(\mathbb{F}_n)^\infty$ has the metric approximation property.
}
\end{ex}

\begin{ex}\label{tori}
{\rm
According to \cite{naito}, the smooth noncommutative 2-torus $\tori$ is isomorphic to the Fr\'echet inductive limit 
\[ \varinjlim \ct\hat{\otimes}_\gamma(M_{p_n}(\C)\oplus M_{q_n}(\C)). \]
We show that it also has the metric approximation property. As a preparation, we verify that the Fr\'echet algebra $\ct\hat{\otimes}_\gamma M_q(\C)$ has the metric approximation property. It suffices to show that $\ct$ has this property since if it had this property with a family $\{\theta_n^{(q)}\}$ of bounded linear maps there, the family $\{\theta_n^{(q)}\otimes I_q\}$ would be the desired one for $\ct\otimes M_q(\C)$, where $I_q$ is the identity map on $M_q(\C)$. For $f\in\ct$, we define the maps $\theta_n^{(q)} : \ct \to \ct$ by
\[ \theta_n^{(q)}(f)=\sum_{|l|\leq n}\hat{f}(l)z^l \quad (n\geq 1), \]
where $\hat{f}(l)$ are the Fourier coefficients and $z\in\ct$ is the canonical generator defined by $z(t)=t \, (t\in T)$. Then it is clear that they are of finite rank. For $f\in\ct$ and $k\in\Z_{\geq 0}$,
\begin{align*}
\|f-\theta_n^{(q)}(f)\|_k&=\|f^{(k)}-(\theta_n(f))^{(k)}\|_\infty \\
&=\sup_{m\in\Z}\left| \widehat{f^{(k)}}(m)-\sum_{|l|\leq n}\hat{f}(l)(2\pi il)^k\delta_l(m)\right| \\
&=\sup_{m\in\Z}\left|\sum_{|l|\geq n+1}\hat{f}(l)(2\pi il)^k\delta_l(m)\right| \\
&=\sup_{m\in\Z, |m|\geq n+1}\left|\hat{f}(m)(2\pi m)^k\right| \\
&\to 0 \quad (n\to\infty),
\end{align*}
where $\delta_l(m)=0 \, (m\neq l), =1(m=l)$, since $\{\hat{f}(l)\}_{l\in\Z}$ is a rapidly decreasing sequence by the hypothesis $f\in\ct$. Hence $\ct$ has the metric approximation property. 

We turn to show briefly that $\tori$ also has this property. For any $x\in\tori$, we define the sequence $\{x_n\}$ by
\[ x_n=e_1^{(n)}xe_1^{(n)}+e_2^{(n)}xe_2^{(n)} \quad (n\geq 1), \]
where $e_j^{(n)}\, (j=1, 2)$ are the projections such that
\[ e_1^{(n)}xe_1^{(n)}\in \ct\otimes M_{p_n}(\C), \quad e_2^{(n)}xe_2^{(n)}\in \ct\otimes M_{q_n}(\C) \]
for any $x\in\tori$ (\cite{naito}). We define the linear maps $\Phi_n$ on $\tori$ by
\[ \Phi_n(x)=\theta_n^{(p_n)}(e_1^{(n)}xe_1^{(n)})+\theta_n^{(q_n)}(e_2^{(n)}xe_2^{(n)}) \]
It is easily seen that $\Phi_n(x)\to x$ with respect to the seminorms on $\tori$ (see \cite{naito}), hence to the metric $d$ as well. Therefore, $\tori$ has the metric approximation property.
}
\end{ex}

By the similar argument for $\ct$, the operation of taking suspension preserves the metric approximation property.

\begin{corollary}\label{cor}
If a $F^*$-algebra $\ore$ has the metric approximation property, so does its suspension $S^\infty\ore$.
\end{corollary}

\begin{proof}
It suffices to show that the $F^*$-algebra
\[
C_0^\infty(0, 1)=\{f\in C^\infty(0, 1) \, | \, f_+^{(n)}(0)=f_-^{(n)}(1)=0 \, (n\in\Z_{\geq 0})\}
\]
has the metric approximation property. For any integer $j\geq 1$, we put
\[
f_j(t)=e^{-\frac{1}{jt(1-t)}}\in C_0^\infty(0, 1).
\]
Let $\{\xi_j\}_{j=1}^\infty$ be the orthogonal family of $C_0^\infty(0, 1)$ obtained by Schmidt orthogonalization of $\{f_j\}$. Then we define the linear maps $\theta_n : C_0^\infty(0, 1) \to C^\infty(0, 1)$ by
\[
\theta_n(f)(t)=\sum_{j=1}^n\langle f | \xi_j\rangle \xi_j(t) \quad (\xi\in C_0^\infty(0, 1), \, t\in (0, 1), \, n\geq 1).
\]
It is easily seen that the images of $\theta_n$ are included in $C_0^\infty(0, 1)$. By the similar argument for $\ct$, we obtain the conclusion.
\end{proof}

For a $F^*$-algebra $\ore$, by $\ore^*$ we denote the set of all bounded linear functionals on $\ore$, where we say a linear functional $\varphi$ on $\ore$ is bounded if and only if
\[
\|\varphi\|=\sup_{a\in\ore\setminus \{0\}}\frac{|\varphi(a)|}{d_{\ore}(a, 0)}<\infty.
\]
Before we proceed to show the lifting problem, we need the following lemma:
\begin{lemma}\label{fact}
Let $\ore, \mathfrak{B}$ be two $F^*$-algebras. Suppose that $\ide$ is an $F^*$-ideal of $\mathfrak{B}$ and that $L, N$ are finite dimensional subspaces of $\ore$ with $L\subset N$. We consider the following diagram of bounded linear maps{\rm :}
\[
\begin{CD}
L @>\iota >\subset > N @>\Psi >> \mathfrak{B} \\
 & &  @|  @VV q V \\
L @>\iota >\subset > N @>>\varphi > \mathfrak{B}/\ide,
\end{CD}
\]
where $q$ and $\iota$ are the quotient map and the natural inclusion respectively, and suppose that
\[
d_{\mathfrak{B}}(q\circ\Psi(a)-\varphi(a), 0)\leq \varepsilon d_\ore(a, 0) \quad (a\in L).
\]
for a positive constant $\e >0$. Then there is a bounded linear map $\varphi' : N \to \mathfrak{B}/\ide$ with the property that
\[
\begin{cases}
\varphi=q\circ\varphi' \\
d_{\mathfrak{B}}(\varphi'(a), 0)\leq d_\ore(a, 0) & (a\in N) \\
d_{\mathfrak{B}}(\varphi'(a)-\Psi(a), 0)\leq 6\varepsilon & (a\in L).
\end{cases}
\]

\begin{proof}
This lemma is an analogy of Lemma 2.5 in \cite{choi}. Let $D'$ and $K$ be the closed unit ball of $L\hat{\otimes}_\gamma\mathfrak{B}$ and $N\hat{\otimes}_\gamma\mathfrak{B}$ respectively, that is,
\begin{align*}
D'&=\{\varphi : \mathfrak{B}\to L \, | \, \|\varphi\|_{L\hat{\otimes}_\gamma\mathfrak{B}^*}\leq 1\}\subset L\hat{\otimes}_\gamma \mathfrak{B}^* \\
K&=\{\varphi : \mathfrak{B}\to N \, | \, \|\varphi\|_{N\hat{\otimes}_\gamma\mathfrak{B}^*}\leq 1\}\subset N\hat{\otimes}_\gamma\mathfrak{B}^*,
\end{align*}
where
\[
\|\varphi\|_{L\hat{\otimes}_\gamma\mathfrak{B}^*}=
\sup_{a\in\mathfrak{B}\setminus\{0\}}\frac{d_\ore(\varphi(a), 0)}{d_{\mathfrak{B}}(a, 0)}
\]
and $\|\cdot\|_{N\hat{\otimes}_\gamma\mathfrak{B}^*}$ is defined by the similar way for $\|\cdot\|_{L\hat{\otimes}_\gamma\mathfrak{B}^*}$, and ${\rm{Aff}}_T(D')$ the set of all affine functions $\psi$ on $D'$ such that $\psi(\alpha \varphi)=\alpha\psi(\varphi)$ for all $\alpha\in T, \varphi\in D'$. It is clear that $\ide$ is a $M$-ideal of $\mathfrak{B}$ and the equality
\[
\mathfrak{B}^*=\ide^\perp\oplus\ide^*
\]
holds as a linear space, where $\ide^\perp$ is the annihilator of $\ide$. Let $e : \mathfrak{B}^*\to\ide^\perp$ be the natural projection and $W$ the image of $N\hat{\otimes}_\gamma\mathfrak{B}^*$ via $1\otimes e$, which is equal to $N\hat{\otimes}_\gamma \ide^\perp$. Then $D'$ is mapped weak$^*$ homeomorphically to $D\subset K$ through the natural embedding $\iota\otimes 1 : L\hat{\otimes}_\gamma\mathfrak{B}^*\to N\hat{\otimes}_\gamma\mathfrak{B}^*$. We may identify the closed unit ball of $N\hat{\otimes}_\gamma\mathfrak{B}^*$ with $F=K\cap W$. We also identify the closed unit ball of $L\hat{\otimes}_\gamma\mathfrak{B}^*$ with $D'\cup W'$, where $W'=(1\otimes e)(L\hat{\otimes}_\gamma\mathfrak{B}^*)=L\hat{\otimes}_\gamma\ide^\perp$. It is verified by the same argument in the proof of Lemma 2.5 in \cite{choi} that
\begin{equation}\label{identify}
(\iota\otimes 1)(1\otimes e)=(1\otimes e)(\iota\otimes 1)
\end{equation}
and
\begin{equation}\label{identify:2}
(\iota\otimes 1)(D'\cap W')=D\cap(\iota\otimes 1)(W')=D\cap W=D\cap F.
\end{equation}
Thus we have the following diagram of restrictions
\begin{equation}\label{restrictions}
\begin{CD}
{\rm{Aff}}_T(D\cap F) @<<< {\rm{Aff}}_T(F) \\
@AAA  @AAA \\
{\rm{Aff}}_T(D) @<<< {\rm{Aff}}_T(K). 
\end{CD}
\end{equation}
Since $1\otimes e : L\hat{\otimes}_\gamma\mathfrak{B}^*\to L\hat{\otimes}_\gamma\ide^\perp$ maps $D'$ onto $D'\cap W$, $D$ is mapped onto $D\cap F$ by (\ref{identify}) and (\ref{identify:2}) and $D$ satisfies the condition of Lemma 2.1 in \cite{choi}. Therefore, with the diagram (\ref{restrictions}), we obtain the conclusion by the same argument of Lemma 2.5 in \cite{choi}.
\end{proof}

\end{lemma}

\begin{proposition}\label{lifting}
Let $\ore, \mathfrak{B}$ be two $F^*$-algebras and $\mathfrak{I}\subset\mathfrak{B}$ an $F^*$-ideal. If $\ore$ is separable and has the metric approximation property, then for any bounded linear map $\varphi : \ore \to \mathfrak{B}/\mathfrak{I}$, there exists a bounded linear map $\Phi : \ore \to \mathfrak{B}$ with the property that $q\circ\Phi=\varphi$, where $q : \mathfrak{B}\to\mathfrak{B}/\mathfrak{I}$ is the quotient map.
\end{proposition}

\begin{proof}
This proof is inspired by that of Theorem 2.6 in \cite{choi}. We fix a sequence $\{a_n\}_{n\in\mathbb{N}}\subset\ore$ dense in $\ore$. We construct recursively the pairs $\{(L_n, \theta_n)\}_{n\in\Z_{\geq 0}}$ which consist of increasing finite dimensional subspaces $L_n\subset\ore$ with $a_n\in L_n$ for any $n\in\Z_{\geq 0}$ and bounded linear maps $\theta_n : \ore \to L_n$ with the property that for any $a\in L_{n-1}$, the inequalities
\[ d_{\ore}(a, \theta_n(a))\leq \frac{1}{2^n} \]
are satisfied. We put $L_0=\{ 0\}$ and $\theta_0=0$. We suppose that for some $n\in\Z_{\geq 0}$ the pairs $(L_0, \theta_0), \cdots , (L_n, \theta_n)$ with the above properties are given. By the approximation property of $\ore$, there exists a bounded linear map $\theta_{n+1} : \ore\to\ore$ such that for each $a\in L_{n}$, the inequality
\[ d_{\ore}(a, \theta_{n+1}(a))\leq \frac{1}{2^{n+1}} \]
holds. We define the subspace $L_{n+1}$ of $\ore$ by
\[ L_{n+1}=L_n+\theta_{n+1}(\ore)+\C a_{n+1}. \]
Then we have the desired pairs $\{(L_n, \theta_n)\}_{n\in\Z_{\geq 0}}$. We note that $\cup_{n\in\Z_{\geq 0}}L_n$ is dense in $\ore$ with respect to Fr\'echet topology.

Next we inductively define a family of bounded linear maps 
\[ \Psi_n : L_n \to \mathfrak{B} \quad (n\in\Z_{\geq 0}) \]
such that
\begin{eqnarray}
q\circ\Psi_n(a)=\varphi(a) \quad (a\in L_n).\label{eq:1}
\end{eqnarray}
Putting $\Psi_0=0$, suppose that for some $n\in\Z_{\geq 0}$, bounded linear maps $\Psi_0, \cdots , \Psi_n$ satisfying (\ref{eq:1}) are constructed. Then we have that for any $a\in L_{n-1}$,
\begin{align*}
d_{\mathfrak{B}/\mathfrak{I}}(q\circ\Psi_n\circ\theta_n(a), \varphi(a)) &=d_{\mathfrak{B}/\mathfrak{I}}(\varphi\circ\theta_n(a), \varphi(a)) \\
&\leq Cd_{\ore}(\theta_n(a), a) \\
&\leq \frac{C}{2^n}.
\end{align*}
By Lemma \ref{fact}, we find a bounded map $\Psi_{n+1} : L_{n+1}\to\mathfrak{B}$ such that $\varphi=q\circ \Psi_{n+1}$ on $L_{n+1}$ and that for any $a\in L_{n-1}$ with
\[ d_{\mathfrak{B}}(\Psi_{n+1}(a), \Psi_n\circ\theta_n(a))\leq \frac{6C}{\,\, 2^n}. \]
Therefore, we compute that
\begin{align*}
d_{\mathfrak{B}}(\Psi_{n+1}(a), \Psi_n(a)) &\leq \frac{6C}{\,\, 2^n}+d_{\mathfrak{B}}(\Psi_n(a), \Psi_n\circ\theta_n(a)) \\
&\leq \frac{6C}{\,\, 2^n}+d_{\ore}(a, \theta_n(a)) \\
&\leq \frac{6C+1}{\,\, 2^n} \quad (a\in L_{n-1}).
\end{align*}
Hence for a fixed integer $n_0\in\Z_{\geq 0}$, we have for all $n\geq n_0$,
\[
d_{\mathfrak{B}}(\Psi_{n+1}(a), \Psi_n(a))\leq \dfrac{6C+1}{\,\, 2^n}.
\]
Thus, for a fixed integer $n_0\in\Z_{\geq 0}$, the family of bounded linear maps $\{\Psi_n\}$ converges to some $\Psi^{(n_0)} : L_{n_0-1} \to \mathfrak{B}$. Therefore, we have the bounded linear map
\[ \Psi : \bigcup_{n\in\Z_{\geq 0}}L_n \to \mathfrak{B} \]
such that $\Psi |_{L_n}=\Psi^{(n)} \, (n\in\Z_{\geq 0})$, and we can extend it to that on the closure of $\cup_{n\in\Z_{\geq 0}}L_n$ which is equal to $\ore$. This completes the proof.
\end{proof}

%%%%%%%%%%%%%%%%%%%%%%%%%%%%%%%%%%%%%%%%%%%%%%%%%%%%%%%
%%%%%%%%%%%%%%%%%% section 5 %%%%%%%%%%%%%%%%%%%%%%%%%%
%%%%%%%%%%%%%%%%%%%%%%%%%%%%%%%%%%%%%%%%%%%%%%%%%%%%%%%

\section{Mayer-Vietoris Exact Sequence}

This section is devoted to proving Mayer-Vietoris exact sequence for the entire cyclic cohomology. We firstly give a short proof of Bott periodicity for the entire cyclic cohomology by using the following Meyer's excision for the entire cyclic theory \cite{meyer}:

\begin{proposition}\label{meyer}
Let
\[
\begin{CD}
0 @>>> K @>i>> P @>q>> Q @>>> 0
\end{CD}
\]
be a short exact sequence of $F^*$-algebras with a bounded linear section $s$ of $q$. Then the following 6-terms exact sequence:
\[
\begin{CD}
HE^{{\rm{ev}}}(Q) @>q^*>> HE^{{\rm{ev}}}(P) @>i^*>> HE^{{\rm{ev}}}(K) \\
@AAA & &  @VVV \\
HE^{{\rm{od}}}(K) @<<i^*< HE^{{\rm{od}}}(P) @<<q^*< HE^{{\rm{od}}}(Q)
\end{CD}
\]
holds.
\end{proposition}
This yields the following fact, which has been already shown by Brodzki and Plymen \cite{brodzki-plymen} using bivariant entire homology and cohomology theory:
\begin{lemma}[Bott periodicity for entire cyclic cohomology]\label{bott}
For a $F^*$-algebra $\ore$,
\[ HE^{{\rm{ev}}}(S^\infty\ore)\simeq HE^{{\rm{od}}}(\ore), \quad HE^{{\rm{od}}}(S^\infty\ore)\simeq HE^{{\rm{ev}}}(\ore). \]
\end{lemma}

\begin{proof}
By the exact sequence cited above, we have the following exact diagram:
\[
\begin{CD}
HE^{{\rm{ev}}}(\ore) @>>> HE^{{\rm{ev}}}(C^\infty\ore) @>>> HE^{{\rm{ev}}}(\mathfrak{I}) \\
@AAA & &  @VVV \\
HE^{{\rm{od}}}(\mathfrak{I}) @<<< HE^{{\rm{od}}}(C^\infty\ore) @<<< HE^{{\rm{od}}}(\ore).
\end{CD}
\]
By Lemma \ref{iso_for_suspension}, we deduce the conclusion.
\end{proof}

In what follows, we show an entire cyclic cohomology version of Mayer-Vietoris exact sequence. Before stating it, we review briefly the fibered product of $F^*$-algebras, which is an noncommutative analogue of the connected sum of two manifolds. Let $\mathfrak{A}_1, \mathfrak{A}_2$ and $\mathfrak{B}$ be $F^*$-algebras and $f_j : \mathfrak{A}_j \to \mathfrak{B} \quad (j=1, 2)$ epimorphisms.

\begin{definition}\label{fiber}
$\{(a_1, a_2)\in \mathfrak{A}_1\oplus\mathfrak{A}_2 \, | \, f_1(a_1)=f_2(a_2) \}$ is called the fibered product of $(\mathfrak{A}_1, \mathfrak{A}_2)$ along $(f_1, f_2)$ over $\mathfrak{B}$, which we denote by $\mathfrak{A}_1\underset{\mathfrak{B}}{\#}\mathfrak{A}_2$. Let $g_j$ be the projections of $\mathfrak{A}_1\underset{\mathfrak{B}}{\#}\mathfrak{A}_2$ onto $\mathfrak{A}_j \, (j=1, 2)$.
\end{definition}

\begin{theorem}[Mayer-Vietoris Exact Sequence for entire cyclic cohomology]
\label{m-v-seq}
In \\
the situation of Definition {\rm{\ref{fiber}}}, suppose that $\mathfrak{B}$ has the metric approximation property and separable. Then we have that the following exact diagram:
\[
\begin{CD}
HE^{{\rm{ev}}}(\ore_1 \underset{\mathfrak{B}}{\#} \ore_2) @>>> HE^{{\rm{od}}}(\mathfrak{B}) @>-f_1^*+f_2^*>> HE^{{\rm{od}}}(\ore_1)\oplus HE^{{\rm{od}}}(\ore_2) \\
@Ag_1^*+g_2^*AA & &  @VVg_1^*+g_2^*V \\
HE^{{\rm{ev}}}(\ore_1)\oplus HE^{{\rm{ev}}}(\ore_2) @<<-f_1^*+f_2^*< HE^{{\rm{ev}}}(\mathfrak{B}) @<<< HE^{{\rm{od}}}(\ore_1 \underset{\mathfrak{B}}{\#} \ore_2)
\end{CD}
\]
holds.
\end{theorem}

\begin{proof}
We write
\[ C=\{ (h_1, h_2)\in C^\infty \ore_1\oplus C^\infty \ore_2 \, | \, f_1\circ (h_1)_+^{(n)}(0)=(-1)^nf_2\circ (h_2)_+^{(n)}(0) \, (n\in\Z_{\geq 0}) \} \]
and define a map $q : C \to \ore_1 \underset{\mathfrak{B}}{\#} \ore_2$ by
\[ q(h_1, h_2)=(h_1(0), h_2(0)). \]
It is easily verified that the following sequence:
\[
\begin{CD}
0 @>>> \mathfrak{I} @>i>> C @>q>> \ore_1 \underset{\mathfrak{B}}{\#} \ore_2 @>>> 0
\end{CD}
\]
is exact, where
\[ \mathfrak{I}=\{ (h_1, h_2)\in C \, | \, h_j(0)=0 \,\, (j=1, 2) \} \]
and $i$ is the canonical inclusion. Then there exists a bounded linear section $s$ of $q$ defined by
\[ s(a_1, a_2)=((1-t)a_1, (1-t)a_2). \quad ((a_1, a_2)\in \ore_1 \underset{\mathfrak{B}}{\#} \ore_2, \, t\in [0, 1]) \]
Then by Proposition \ref{meyer}, we have the following exact diagram:
\[
\begin{CD}
HE^{{\rm{ev}}}(\ore_1 \underset{\mathfrak{B}}{\#} \ore_2) @>>> HE^{{\rm{ev}}}(C) @>>> HE^{{\rm{ev}}}(\mathfrak{I}) \\
@AAA & &  @VVV \\
HE^{{\rm{od}}}(\mathfrak{I}) @<<< HE^{{\rm{od}}}(C) @<<< HE^{{\rm{od}}}(\ore_1 \underset{\mathfrak{B}}{\#} \ore_2).
\end{CD}
\]
Moreover, repeating the argument cited above, $\mathfrak{I}$ is smoothly homotopic to \\
$S^\infty\ore_1\oplus S^\infty\ore_2$. More precisely, we define the map
\[ r : \mathfrak{I} \to S^\infty\ore_1\oplus S^\infty\ore_2 \]
by
\[ r(h_1, h_2)(t)=(h_1(e^{1-1/t}), h_2(e^{1-1/t})) \quad ((h_1, h_2)\in C) \]
and let $i : S^\infty\ore_1\oplus S^\infty\ore_2 \to \mathfrak{I}$ be the natural inclusion. It follows by the same argument discussed above that since the functions $t\mapsto h_j(e^{1-1/t})$ are in $S^\infty \ore_j \, (j=1, 2)$ and using the maps
\begin{align*}
G_1 &: \mathfrak{I} \to C^\infty ([0, 1], \mathfrak{I}) \\
G_2 &: S^\infty\ore_1\oplus S^\infty\ore_2 \to C^\infty ([0, 1], S^\infty\ore_1\oplus S^\infty\ore_2 )
\end{align*}
defined by
\[ (G_j)_s(h_1, h_2)(t)=(h_1(se^{1-1/t}+(1-s)t), h_2(se^{1-1/t}+(1-s)t)) \quad (j=1, 2), \]
$\mathfrak{I}$ is smoothly homotopic to $S^\infty\ore_1\oplus S^\infty\ore_2$. Hence we conclude that
\[ HE^*(\mathfrak{I})\simeq HE^*(S^\infty\ore_1)\oplus HE^*(S^\infty\ore_2). \]
Now we define the map $\Psi : C \to S^\infty\mathfrak{B}$ by
\[ \Psi (h_1, h_2)(t)=
\begin{cases}
f_1\circ h_1(1-2t) & (t\in [0, 1/2]) \\
f_2\circ h_2(2t-1) & (t\in [1/2, 1]).
\end{cases}
\]
We have to verify that it is well-defined. Since $f_1\circ h_1(0)=f_2\circ h_2(0)$ by the definition of $C$, it is continuous at $t=1/2$. For $n=1$, we compute that
\begin{align*}
\lim_{t\to 1/2+0}\dfrac{\Psi(h_1, h_2)(t)-\Psi(h_1, h_2)(1/2)}{t-1/2}
&=\lim_{t\to 1/2+0}\dfrac{f_2\circ h_2(2t-1)-f_2\circ h_2(0)}{t-1/2} \\
&=f_2 \left( \lim_{t\to 1/2+0}\dfrac{h_2(2t-1)-h_2(0)}{t-1/2} \right) \\
&=2f_2 \left( \lim_{\e \to 0+}\dfrac{h_2(\e )-h_2(0)}{\e}\right) \\
&=2f_2\circ (h_2)_+^{(1)}(0)
\end{align*}
and similarly, we compute that
\begin{align*}
\lim_{t\to 1/2-0}\dfrac{\Psi(h_1, h_2)(t)-\Psi(h_1, h_2)(1/2)}{t-1/2}
&=\lim_{t\to 1/2-0}\dfrac{f_1\circ h_1(1-2t)-f_1\circ h_1(0)}{t-1/2} \\
&=-2f_1\left( \lim_{\e \to 0+}\dfrac{h_1(\e )-h_1(0)}{\e}\right) \\
&=-2f_1\circ (h_1)_+^{(1)}(0)=2f_2\circ (h_2)_+^{(1)}(0).
\end{align*}
Thus $\Psi(h_1, h_2)$ is differentiable once at $t=1/2$. Suppose that it is differentiable $n$-times at $t=1/2$. Here we note that
\[
\Psi^{(n)}(h_1, h_2)(t)=
\begin{cases}
(-2)^nf_1\circ h_1^{(n)}(1-2t) & t\in (0, 1/2) \\
2^nf_2\circ h_2^{(n)}(2t-1) & t\in (1/2, 1)
\end{cases}
\]
and that
\[ \Psi^{(n)}(h_1, h_2)(1/2)=(-2)^nf_1\circ (h_1)_+^{(n)}(0)=2^nf_2\circ (h_2)_+^{(n)}(0) \]
by our hypothesis of induction. Then we compute that
\begin{align*}
&\lim_{t\to 1/2+0}\dfrac{\Psi(h_1, h_2)^{(n)}(t)-\Psi(h_1, h_2)^{(n)}(1/2)}{t-1/2} \\
&=2^n\lim_{t\to 1/2+0}\dfrac{f_2\circ h_2^{(n)}(2t-1)-f_2\circ (h_2)_+^{(n)}(0)}{t-1/2} \\
&=2^nf_2\left(\lim_{t\to 1/2+0}\dfrac{h_2^{(n)}(2t-1)-(h_2)_+^{(n)}(0)}{t-1/2}\right) \\
&=2^n\cdot 2f_2\left(\lim_{\e\to 0+}\dfrac{h_2^{(n)}(\e )-(h_2)_+^{(n)}(0)}{\e}\right) \\
&=2^{n+1}f_2\circ (h_2)_+^{(n+1)}(0).
\end{align*}
Alternatively, we compute that
\begin{align*}
&\lim_{t\to 1/2-0}\dfrac{\Psi(h_1, h_2)^{(n)}(t)-\Psi(h_1, h_2)^{(n)}(1/2)}{t-1/2} \\
&=(-2)^n\lim_{t\to 1/2-0}\dfrac{f_1\circ h_1^{(n)}(1-2t)-f_1\circ (h_1)_+^{(n)}(0)}{t-1/2} \\
&=(-2)^n\cdot (-2)f_1\left(\lim_{\e\to 0+}\dfrac{h_1^{(n)}(\e )-(h_1)_+^{(n)}(0)}{\e}\right) \\
&=(-2)^{n+1}f_1\circ (h_1)_+^{(n+1)}(0)=2^{n+1}f_2\circ (h_2)_+^{(n+1)}(0).
\end{align*}
Therefore, $\Psi(h_1, h_2)$ is differentiable $(n+1)$-times for each $(h_1, h_2)\in C$, which ends the process of induction so that $\Psi$ is well-defined. Since $f_1$ and $f_2$ are surjective, it is easily verified that $\Psi$ is surjective. In fact, we canonically can lift them on $S^\infty\ore_j$, which are denoted by $f_j\, (j=1, 2)$. Now given a $h\in S^\infty\mathfrak{B}$, we find $\widetilde{h_j}\in S^\infty\ore_j$ with 
\[
f_j(\widetilde{h_j}(t))=h(t) \quad (j=1, 2).
\]
Putting
$
h_1(t)=\widetilde{h_1}((1-t)/2), \, h_2(t)=\widetilde{h_2}((1+t)/2) \, (0\leq t \leq 1),
$
we then check that
\begin{align*}
\Psi(h_1, h_2)(t)&=
\begin{cases}
f_1(h_1(1-2t)) & (0\leq t\leq 1/2) \\
f_2(h_2(2t-1)) & (1/2\leq t\leq 1)
\end{cases}
\\
&=
\begin{cases}
f_1\left(\widetilde{h_1}\left(1-2\dfrac{1-t}{2}\right)\right) & (0\leq t\leq 1/2) \\
f_2\left(\widetilde{h_2}\left(2\dfrac{1+t}{2}-1\right)\right) & (1/2\leq t\leq 1)
\end{cases}
\\
&=
\begin{cases}
f_1(\widetilde{h_1}(t)) & (0\leq t\leq 1/2) \\
f_2(\widetilde{h_2}(t)) & (1/2\leq t\leq 1)
\end{cases}
\\
&=h(t),
\end{align*}
which implies that $\Psi$ is surjective. As it is clear that its kernel is $C^\infty \ide_1\oplus C^\infty \ide_2$, where
\[ \ide_j={\rm{Ker}} f_j \quad (j=1, 2), \]
we obtain the following short exact sequence:
\begin{eqnarray}\label{sequence:1}
\begin{CD}
0 @>>> C^\infty\ide_1\oplus C^\infty\ide_2 @>>> C @>\Psi>> S^\infty\mathfrak{B} @>>> 0.
\end{CD}
\end{eqnarray}
Since $\mathfrak{B}$ has the metric approximation property, so does $S^\infty\mathfrak{B}$ by Corollary \ref{cor}. Writing $\mathfrak{J}=C^\infty\ide_1\oplus C^\infty\ide_2$, the inverse map 
\[ \overline{\Psi}^{-1} : S^\infty\mathfrak{B} \to C/\mathfrak{J} \]
of the isomorphism $\overline{\Psi}$ induced by $\Psi$ has a bounded lift
\[
\widetilde{\overline{\Psi}}^{-1} : S^\infty\mathfrak{B}\to C
\]
satisfying $\widetilde{\overline{\Psi}}^{-1}\circ q=\overline{\Psi}^{-1}$ by Proposition \ref{lifting} since $\overline{\Psi}$ preserves each seminorms, where $q$ is the quotient map from $C$ onto $C/\mathfrak{J}$. Hence it is verified that $\widetilde{\overline{\Psi}}^{-1}$ is a bounded linear section of $\Psi$ since we compute that
\[
\Psi\circ\widetilde{\overline{\Psi}}^{-1}=\overline{\Psi}\circ q\circ\widetilde{\overline{\Psi}}^{-1}=\overline{\Psi}\circ\overline{\Psi}^{-1}=id_{S^\infty\mathfrak{B}}.
\]
Therefore, we apply the above exact sequence (\ref{sequence:1}) to Proposition \ref{meyer} to obtain the following exact diagram:
\[
\begin{CD}
HE^{{\rm{ev}}}(S^\infty\mathfrak{B}) @>>> HE^{{\rm{ev}}}(C) @>>> HE^{{\rm{ev}}}(C^\infty\ide_1\oplus C^\infty\ide_2) \\
@AAA & &  @VVV \\
HE^{{\rm{od}}}(C^\infty\ide_1\oplus C^\infty\ide_2) @<<< HE^{{\rm{od}}}(C) @<<< HE^{{\rm{od}}}(S^\infty\mathfrak{B}).
\end{CD}
\]
Since $HE^*(C^\infty\ide_1\oplus C^\infty\ide_2)=0$, we have that
\begin{align*}
HE^{{\rm{ev}}}(C)&\simeq HE^{{\rm{ev}}}(S^\infty\mathfrak{B}) \simeq HE^{{\rm{od}}}(\mathfrak{B}) \\
HE^{{\rm{od}}}(C)&\simeq HE^{{\rm{od}}}(S^\infty\mathfrak{B}) \simeq HE^{{\rm{ev}}}(\mathfrak{B})
\end{align*}
by the Bott periodicity (Lemma \ref{bott}).

Summing up, we get the desired exact diagram in what follows:
\begin{align}
\begin{CD}
HE^{{\rm{ev}}}(\ore_1 \underset{\mathfrak{B}}{\#} \ore_2) @>>> HE^{{\rm{od}}}(\mathfrak{B}) @>>> HE^{{\rm{od}}}(\ore_1)\oplus HE^{{\rm{od}}}(\ore_2) \\
@AAA & &  @VVV \\
HE^{{\rm{ev}}}(\ore_1)\oplus HE^{{\rm{ev}}}(\ore_2) @<<< HE^{{\rm{ev}}}(\mathfrak{B}) @<<< HE^{{\rm{od}}}(\ore_1 \underset{\mathfrak{B}}{\#} \ore_2).
\end{CD}\label{diagram}
\end{align}
We consider the restriction $\Phi : S^\infty\ore_1\oplus S^\infty\ore_2 \to S^\infty\mathfrak{B}$ of $\Psi$. We see that it is $C^\infty$-homotopic to $\Pi : S^\infty\ore_1\oplus S^\infty\ore_2\to\mathfrak{B}$ defined by
\[
\Pi(h_1, h_2)(t)=-\chi_{[0, 1/2]}(t)(f_1\circ h_1)(t)+\chi_{[1/2, 1]}(t)(f_2\circ h_2)(t)
\]
for $(h_1, h_2)\in S^\infty\ore_1\oplus S^\infty\ore_2, t\in [0, 1]$. To see this, we note that for a Fr\'echet continuous homomorphism $f : \ore_1\underset{\mathfrak{B}}{\#}\ore_2\to S^\infty\mathfrak{B}$, we have
\[
\overline{f}^*=-f^* : HE^*(S^\infty\mathfrak{B})\to HE^*(\ore_1\underset{\mathfrak{B}}{\#}\ore_2)
\]
by \cite{meyer}, where $\overline{f} : \ore_1\underset{\mathfrak{B}}{\#}\ore_2\to S^\infty\mathfrak{B}$ is the homomorphism defined by
\[
\overline{f}(a)(t)=f(a)(1-t) \quad (a\in\ore, t\in [0, 1]).
\]
Indeed, we prepare the map $\Theta : S^\infty\ore_1\oplus S^\infty\ore_2 \to C^\infty([0, 1], S^\infty\mathfrak{B})$ defined by
\[
\Theta_s(h_1, h_2)(t)=
\begin{cases}
f_1\circ h_1(1-2t/(1+s)) & (0\leq t\leq 1/2) \\
f_2\circ h_2(2t/(1+s)-(1-s)/(1+s)) & (1/2\leq t\leq 1).
\end{cases}
\]
so that it is a smooth homotopy between $\Psi$ and the homomorphism given by
\[
(h_1, h_2)\mapsto \left(t\mapsto \chi_{[0, 1/2]}(t)(f_1\circ h_1)(1-t)+\chi_{[1/2, 1]}(t)(f_2\circ h_2)(t)\right).
\]
Therefore, we have the homotopy equivalence of $\Psi$ and $\Pi$. Considering the following commutative diagram:
\[
\begin{CD}
HE^*(C) @>>> HE^*(S^\infty\ore_1\oplus S^\infty\ore_2) \\
@A\simeq A\Psi^* A @| \\
HE^*(S^\infty\mathfrak{B}) @>>\Gamma^*=\Pi^*> HE^*(S^\infty\ore_1\oplus S^\infty\ore_2 )
\end{CD}
\]
we conclude that the right upper horizonal map and the left lower horizonal map in the diagram (\ref{diagram}) are both $\Pi^*=-f_1^*+f_2^*$. Finally, since the following diagram
\[
\begin{CD}
HE^*(S^\infty\ore_1)\oplus HE^*(S^\infty\ore_2) @>>> HE^*(\ore_1\underset{\mathfrak{B}}{\#}\ore_2) \\
@V\simeq VV  @| \\
HE^*(\ore_1)\oplus HE^*(\ore_2) @>>g_1^*+g_2^*> HE^*(\ore_1\underset{\mathfrak{B}}{\#}\ore_2)
\end{CD}
\]
is commutative, the vertical maps in the diagram (\ref{diagram}) are both $g_1^*+g_2^*$. This completes the proof.
\end{proof}

%%%%%%%%%%%%%%%%%%%%%%%%%%%%%%%%%%%%%%%%%%%%%%%%%%%%%%%%%%%%
%%%%%%%%%%%%%%%%%%%%%% section 6 %%%%%%%%%%%%%%%%%%%%%%%%%%%
%%%%%%%%%%%%%%%%%%%%%%%%%%%%%%%%%%%%%%%%%%%%%%%%%%%%%%%%%%%%

\section{The Entire Cyclic Cohomology of Noncommutative 3-spheres}

In \cite{bhms}, Heegaard-type quantum 3-spheres with 3-parameters are constructed as $C^*$-algebras. With their construction in mind, we define noncommutative 3-spheres in the framework of $F^*$-algebras as follows; given an irrational number $\theta$ with $0<\theta <1$, let $\tori$ be the smooth noncommutative 2-torus with unitary generators $u_\theta, v_\theta$ subject to $u_\theta v_\theta =e^{2\pi i\theta}v_\theta u_\theta$. There exists an isomorphism $\gamma_\theta : T_{-\theta}^2 \to \tori $ satisfying
\[ \gamma_\theta (u_{-\theta})=v_\theta, \quad \gamma_\theta (v_{-\theta})=u_\theta \]
by their universality. We consider the following two $F^*$-crossed products:
\[ (D^2\times S^1)_\theta = \toep \rtimes_{\alpha_\theta} \Z , \quad (D^2\times S^1)_{-\theta} = \toep \rtimes_{\alpha_{-\theta}}\Z \]
defined before. We define two epimorphisms $f_j \, (j=1, 2)$ such as
\[ f_1 : (D^2\times S^1)_\theta \to \tori, \quad f_2 : (D^2\times S^1)_{-\theta} \to \tori \]
by $f_1=\widetilde{q}_+, \, f_2=\gamma_\theta\circ\widetilde{q}_-$, where $\widetilde{q}_\pm$ are the epimorphisms from $\toep\rtimes_{\alpha_{\pm\theta}}\Z$ onto $\ct\rtimes_{\overline{\alpha}_{\pm\theta}}\Z=T^2_{\pm\theta}$ respectively.

\begin{definition}
Given an irrational number $\theta$, the noncommutative 3-sphere $S^3_\theta$ is defined by the fibered product $(D^2\times S^1)_\theta \underset{\tori}{\#} (D^2\times S^1)_{-\theta}$ of $((D^2\times S^1)_\theta\, (D^2\times S^1)_{-\theta})$ along $(f_1, f_2)$ over $\tori$.
\end{definition}

First of all, we compute the entire cyclic cohomology of $(D^2\times S^1)_\theta$. We note that the isomorphism $\ct\rtimes_{\overline{\alpha}_\theta}\Z\simeq\tori$ holds and that by Lemma 4.3 in \cite{naito}, we have
\[
HE^*(\ct\rtimes_{\overline{\alpha}_\theta}\Z)\simeq HE^*(\tori)=HP^*(\tori),
\]
where $HP^*$ is the functor of periodic cyclic cohomology. According to Connes \cite{connes}, we know the generators of $HP^*(\tori )$ as follows:
\begin{align*}
HP^{{\rm{ev}}}(\tori )&= \C [\tau_\theta ]\oplus \C [\tau'_\theta ], \\
HP^{{\rm{od}}}(\tori )&= \C [\tau_\theta^{(1)}]\oplus \C [\tau_\theta^{(2)}], 
\end{align*}
where $\tau_\theta$ is the unique normalized trace on $\tori$ and
\begin{align*}
\tau'_\theta(a_0, a_1, a_2)&=\tau_\theta(a_0(\delta_\theta^{(1)}(a_1)\delta_\theta^{(2)}(a_2)-\delta_\theta^{(2)}(a_1)\delta_\theta^{(1)}(a_2))) \\
\tau_\theta^{(j)}(a_0, a_1)&=\tau_\theta(a_0\delta_\theta^{(j)}(a_1)) \quad (j=1, 2),
\end{align*}
where $\delta_\theta^{(j)}$ are the derivations on $\tori$ such that
\[ \delta_\theta^{(1)}(u_\theta )=2\pi iu_\theta, \, \delta_\theta^{(1)}(v_\theta )=0, \, \delta_\theta^{(2)}(u_\theta )=0, \, \delta_\theta^{(2)}(v_\theta )=2\pi iv_\theta . \]

\begin{proposition}\label{d2theta}
\[
HE^{{\rm{ev}}}((D^2\times S^1)_\theta)=\C [\tau_\theta'\circ\widetilde{q}], \quad HE^{{\rm{od}}}((D^2\times S^1)_\theta)=\C[\tau_\theta^{(1)}\circ\widetilde{q}].
\]
\end{proposition}

\begin{proof}
We remember the following short exact sequence:
\[
\begin{CD}
0 @>>> \comp^\infty\rtimes_{\alpha_\theta}\Z @>\widetilde{i}>> (D^2\times S^1)_\theta @>\widetilde{q}>> \ct\rtimes_{\overline{\alpha}_\theta}\Z @>>> 0
\end{CD}
\]
appeared in Corollary \ref{seq:2}. Hence we apply the above exact sequence to Proposition $\ref{meyer}$ to obtain the following exact diagram:
\[
\begin{CD}
HE^{{\rm{ev}}}(\ct\rtimes_{\overline{\alpha}_\theta}\Z) @>\widetilde{q}^*>> HE^{{\rm{ev}}}((D^2\times S^1)_\theta) @>\widetilde{i}^*>> HE^{{\rm{ev}}}(\comp^\infty\rtimes_{\alpha_\theta}\Z) \\
@AAA & &  @VVV \\
HE^{{\rm{od}}}(\comp^\infty\rtimes_{\alpha_\theta}\Z) @<<\widetilde{i}^*< HE^{{\rm{od}}}((D^2\times S^1)_\theta) @<<\widetilde{q}^*< HE^{{\rm{od}}}(\ct\rtimes_{\overline{\alpha}_\theta}\Z).
\end{CD}
\]

Alternatively, we have by Corollary \ref{stableconti} and \cite{mathai-stevenson} that
\begin{align*}
HE^*(\comp^\infty\rtimes_{\alpha_\theta}\Z ) &\simeq HE^*(\comp^\infty\hat{\otimes}_\gamma\ct) \\
&=H^{{\rm{DR}}}_*(T ; \C),
\end{align*}
which implies that
\[
HE^{{\rm{ev}}}(\comp^\infty\rtimes_{\alpha_\theta}\Z)\simeq \C, \quad HE^{{\rm{od}}}(\comp^\infty\rtimes_{\alpha_\theta}\Z)\simeq \C .
\]
Therefore, we have the following exact diagram:
\begin{equation}
\begin{CD}\label{vertical}
\C^2 @>\widetilde{q}^*>> HE^{{\rm{ev}}}((D^2\times S^1)_\theta) @>\widetilde{i}^*>> \C\, \\
@AAA & & @VVV \\
\C\, @<<\widetilde{i}^*< HE^{{\rm{od}}}((D^2\times S^1)_\theta) @<<\widetilde{q}^*< \C^2.
\end{CD}
\end{equation}
We note that there exists an element $[(\widetilde{\psi}_{2k+1})]\in HE^{{\rm{od}}}((D^2\times S^1)_\theta)$ with the property that
\[
(\widetilde{\psi}_{2k+1})=(\widetilde{\psi}, 0, 0, \cdots ),
\]
and
\[
B\widetilde{\psi}=\tau_\theta\circ \widetilde{q}, \quad b\widetilde{\psi}=0,
\]
where $b, B=AB_0$ are the operations defined by Connes \cite{connes}. Indeed, we define $\widetilde{\psi}$ by
\[
\widetilde{\psi}(x, y)=\tau_\theta\circ\widetilde{q}(x\widetilde{\delta_\theta}^{(2)}(y)) \quad (x, y\in (D^2\times S^1)_\theta),
\]
where $\widetilde{\delta_\theta}^{(2)}$ is the derivation on $(D^2\times S^1)_\theta$ induced by
\[
\widetilde{\delta_\theta}^{(2)}\left(\sum_{n\in\Z}A_nU_\theta^n\right)
=\sum_{n\in\Z}2\pi i\theta nA_nU_\theta^n
\]
for any $\sum_{n\in\Z}A_nU_\theta^n\in\toep [\Z]$. We note that $\widetilde{\delta_\theta}^{(2)}$ is Fr\'echet continuous since
\begin{align*}
\left\|\widetilde{\delta_\theta}^{(2)}\left(\sum_{n\in\Z}A_nU_\theta^n\right)\right\|_{p, q, r, s}&=\sup_{n\in\Z}(1+n^2)^p\left\|2\pi i\theta nA_n\right\|_{q, r, s} \\
&\leq 2\pi\theta \sup_{n\in\Z}(1+n^2)^{p+1}\left\|A_n\right\|_{q, r, s} \\
&=2\pi\theta\left\|\sum_{n\in\Z}A_nU_\theta^n\right\|_{p+1, q, r, s}.
\end{align*}
for any $p, q, r, s\in\Z_{\geq 0}$. In this case, let $1\in \toep\rtimes_{\alpha_\theta}\Z$ be the unit. It is clear that $\widetilde{\delta_\theta}^{(2)}(1)=0$. Then by the definition of $b$ and $B$, we have that
\begin{align*}
B\widetilde{\psi}(x)&=\widetilde{\psi}(1, x)+\widetilde{\psi}(x, 1) \\
&=\tau_\theta\circ\widetilde{q}(x\widetilde{\delta_\theta}^{(2)}(1))+\tau_\theta\circ\widetilde{q}(1\widetilde{\delta_\theta}^{(2)}(x)) \\
&=\tau_\theta\circ\widetilde{q}(\widetilde{\delta_\theta}^{(2)}(x)) \quad (x\in (D^2\times S^1)_\theta).
\end{align*}
We note that for any $f\in \ct\rtimes_{\overline{\alpha}_\theta}\Z$,
\[
\tau_\theta(f)=\int_T f(0)(t)dt.
\]
Thus, we obtain that
\begin{align*}
\tau_\theta\circ\widetilde{q}(\widetilde{\delta_\theta}(x))&=\int_T \widetilde{q}(\widetilde{\delta_\theta}(x))(0)(t)dt \\
&=\int_T q(x(0))(t)dt=\tau_\theta\circ\widetilde{q}(x)
\end{align*}
for any $x\in (D^2\times S^1)_\theta$, which implies that $\widetilde{q}^*[\tau_\theta]=0$. Hence, $\ker \widetilde{q}^*\neq 0$ so that the left vertical map of (\ref{vertical}) is not $0$, therefore, injective.

Similarly, we show that the right vertical map is also injective. Since $\theta$ is an irrational number, the set $\{e^{2\pi i\theta n}\in\C \, | \, n\in\Z\}$ is dense in $T$. Hence, for all $r\in [0, 1]$, there exists an sequence $\{N_j\}_j\subset\Z$ with $|\{\theta N_j\}-r|\to 0$ as $j\to\infty$, where
\[
\{ x\}=x-\max_{x\geq k, \, k\in\Z}k \quad (x\in\R).
\]
We consider the family $\{U_{\theta N_j}\}$ of unitary operators on $H^2$. Since we see that for any $\xi\in H^2$,
\begin{align*}
\|(U_{\theta N_j}-U_{\theta N_k})\xi\|_{H^2}^2&=\|(U_{\theta}^{N_j}-U_{\theta}^{N_k})\xi\|_{H^2}^2 \\
&=\int_T |\xi(e^{2\pi i\theta N_j}t)-\xi(e^{2\pi i\theta N_k}t)|^2 dt \\
&=\int_T |\xi(e^{2\pi i\theta (N_j-N_k)}t)-\xi(t)|^2 dt \to 0 \quad (j, k\to \infty)
\end{align*}
by the Lebesgue dominated convergence theorem, we obtain that $\{U_{\theta N_j}\}$ has the strong limit $U_r$. It is easily seen that $U_r\xi(t)=\xi(e^{2\pi ir}t) \, (\xi\in H^2, \, t\in T)$. Moreover, we define the operator $h_\theta$ on $H^2$ by
\[
h_\theta\xi(t)=2\pi\sum_{j=0}^\infty \{j\theta\}c_jt^j \quad \left(\xi(t)=\sum_{j=0}^\infty c_jt^j\in H^2\right).
\]
Since $0\leq\{j\theta\}\leq 1$, it is easily verified that $h_\theta$ is a bounded self-adjoint positive operator on $H^2$ and $U_{\theta r}=e^{irh_\theta}$ for $r\in [0, 1]$ by Stone's theorem. Taking again a family $\{N_j\}_{j\in\Z_{\geq 0}}\subset\Z$ with $|e^{2\pi i\theta N_j}-e^{2\pi ir}|\to 0$ as $j\to\infty$, we have that
\begin{align*}
&\|\alpha_{\theta N_j}(x)\xi-\alpha_{\theta N_k}(x)\xi\|_{H^2} \\
&=\|U_{\theta N_j}xU_{-\theta N_j}\xi-U_{\theta N_k}xU_{-\theta N_k}\xi\|_{H^2} \\
&\leq\|U_{\theta N_j}x(U_{-\theta N_j}-U_{-\theta N_k})\xi\|_{H^2}+\|(U_{\theta N_j}-U_{\theta N_k})xU_{-\theta N_k}\xi\|_{H^2} \\
&\to 0 \quad (x\in\toep, \, \xi\in H^2)
\end{align*}
since the operation of product is strongly continuous. Therefore, it follows that $\alpha_r(x)=U_rxU_{-r}$ for $x\in\mathbb{B}(H^2)$. We write
\[
\widetilde{\delta_\theta}^{(1)}(x)=h_\theta x-xh_\theta ={{\rm{ad}}}(h_\theta)(x) \quad (x\in\toep\rtimes_{\alpha_\theta}\Z)
\]
so that
\[
e^{ir\widetilde{\delta_\theta}^{(1)}}=e^{ir{{\rm{ad}}}(h_\theta)}=\alpha_{\theta r} \quad (r\in [0, 1]).
\]
We now extend the homomorphism $\widetilde{q} : \toep\rtimes_{\alpha_\theta}\Z \to \ct\rtimes_{\overline{\alpha}_\theta}\Z$ to that from the strong closure of $\toep\rtimes_{\alpha_\theta}\Z$ onto that of $\ct\rtimes_{\overline{\alpha}_\theta}\Z$ faithfully acting on $L^2(T)$ because of the simplicity of $\tori=\ct\rtimes_{\overline{\alpha}_\theta}\Z$, that is, that from $\mathbb{B}(H^2)$ onto $L^\infty(T)\rtimes_{\overline{\alpha}_\theta}\Z$. We also extend the trace $\tau_\theta$ on $\tori$ to that on $L^\infty(T)\rtimes_{\overline{\alpha}_\theta}\Z$. We use the same letters for their extensions. Then, we have that $\widetilde{q}\circ\widetilde{\delta_\theta}^{(1)}=\delta_\theta^{(1)}\circ\widetilde{q}$ on $\mathbb{B}(H^2)$. Under the above preparation, we define the linear functional $\varphi_0$ on $\toep\rtimes_{\alpha_\theta}\Z$ by
\[
\varphi_0(a)=-\tau_\theta\circ\widetilde{q}(ah_\theta) \quad (a\in \toep\rtimes_{\alpha_\theta}\Z).
\]
Then we compute that
\begin{align*}
(b\varphi_0)(a, b)&=\varphi_0(ab)-\varphi_0(ba) \\
&=-\tau_\theta\circ\widetilde{q}(abh_\theta)+\tau_\theta\circ\widetilde{q}(bah_\theta) \\
&=-\tau_\theta(\widetilde{q}(a)\widetilde{q}(b)\widetilde{q}(h_\theta))
+\tau_\theta(\widetilde{q}(b)\widetilde{q}(a)\widetilde{q}(h_\theta)) \\
&=\tau_\theta(\widetilde{q}(a)\widetilde{q}(h_\theta)\widetilde{q}(b)
-\widetilde{q}(a)\widetilde{q}(b)\widetilde{q}(h_\theta)) \\
&=\tau_\theta(\widetilde{q}(a)\widetilde{q}(h_\theta b-bh_\theta)) \quad (a, b\in\toep\rtimes_{\alpha_\theta}\Z).
\end{align*}
By the definition of $\widetilde{\delta_\theta}^{(1)}$ and $\widetilde{q}\circ\widetilde{\delta_\theta}^{(1)}=\delta_\theta^{(1)}\circ\widetilde{q}$, we have that
\begin{align*}
(b\varphi_0)(a, b)&=\tau_\theta(\widetilde{q}(a)\widetilde{q}\circ\widetilde{\delta_\theta}^{(1)}(b)) \\
&=\tau_\theta(\widetilde{q}(a)\delta_\theta^{(1)}\circ\widetilde{q}(b))
=(\tau_\theta^{(1)}\circ\widetilde{q})(a, b)
\end{align*}
for any $a, b\in\toep\rtimes_{\alpha_\theta}\Z$. Therefore, we obtain that
\[
(b+B)[(\varphi_0, 0, \cdots )]=[(\tau_\theta^{(1)}\circ\widetilde{q}, 0, \cdots )],
\]
which means that $[\tau_\theta^{(1)}\circ\widetilde{q}]=0\in HE^{{\rm{od}}}(\toep\rtimes_{\alpha_\theta}\Z)$. Hence, we have $\ker \widetilde{q}^*\neq 0$ so that the right vertical map of (\ref{vertical}) is also injective.

Summing up, we obtain the following exact diagram:

\begin{equation}\label{diag:1}
\begin{CD}
\C^2 @>\widetilde{q}^*>> HE^{{\rm{ev}}}((D^2\times S^1)_\theta) @>0 >> \C\, \\
@AAA & & @VVV \\
\C\, @<<0 < HE^{{\rm{od}}}((D^2\times S^1)_\theta) @<<\widetilde{q}^*< \C^2.
\end{CD}
\end{equation}
to conclude that
\[
HE^*((D^2\times S^1)_\theta)\simeq \C^2 /\C \simeq \C
\]
as required. Moreover, we easily seen that $\widetilde{q}^*\neq 0$. Hence $\widetilde{q}^*[\tau_\theta']=[\tau_\theta'\circ\widetilde{q}]$ and $\widetilde{q}^*[\tau_\theta^{(2)}]=[\tau_\theta^{(2)}\circ\widetilde{q}]$ are the generators of corresponding entire cyclic cohomology.
\end{proof}

We need the following lemma to end up the main result:

\begin{lemma}\label{trace}
We have the following equalities:
\begin{enumerate}
\item $\tau_\theta\circ\gamma_\theta=\tau_{-\theta}$ and $\tau_\theta' \circ \gamma_\theta = -\tau_{-\theta}'$,
\item $\tau_\theta^{(1)}\circ \gamma_\theta = \tau_{-\theta}^{(2)}$ and $\tau_\theta^{(2)}\circ \gamma_\theta = \tau_{-\theta}^{(1)}$.
\end{enumerate}
\end{lemma}

\begin{proof}
Since $\tau_\theta\circ\gamma_\theta$ is a normalized trace on $T^2_{-\theta}$, it follows by uniqueness that $\tau_\theta\circ\gamma_\theta = \tau_{-\theta}$. We firstly verify that
\[ \delta_\theta^{(1)}\circ \gamma_\theta =\delta_{-\theta}^{(2)}, \quad \delta_\theta^{(2)}\circ \gamma_\theta =\delta_{-\theta}^{(1)}. \]
In fact, it is sufficient to verify these equalities for generators. We compute that
\begin{align*}
\delta_\theta^{(j)}\circ\gamma_\theta(u_{-\theta})&=\delta_\theta^{(j)}(v_\theta)=
\begin{cases}
0 & (j=1) \\
2\pi iv_\theta & (j=2)
\end{cases}
\\
\delta_\theta^{(j)}\circ\gamma_\theta(v_{-\theta})&=\delta_\theta^{(j)}(u_\theta)=
\begin{cases}
2\pi iu_\theta & (j=1) \\
0 & (j=2).
\end{cases}
\end{align*}
We then deduce that
\begin{align*}
&\tau_\theta'(\gamma_\theta(b_0), \gamma_\theta(b_1), \gamma_\theta(b_2)) \\
&=\tau_\theta(\gamma_\theta(b_0)((\delta_\theta^{(1)}\circ\gamma_\theta(b_1))(\delta_\theta^{(2)}\circ\gamma_\theta(b_2))-(\delta^{(2)}\circ\gamma_\theta(b_1))(\delta_\theta^{(1)}\gamma_\theta(b_2)))) \\
&=\tau_\theta(\gamma_\theta(b_0(\delta_{-\theta}^{(2)}(b_1)\delta_{-\theta}^{(1)}(b_2)-\delta_{-\theta}^{(1)}(b_1)\delta_{-\theta}^{(2)}(b_2))) \\
&=-\tau_\theta\circ\gamma_\theta(b_0(\delta_{-\theta}^{(1)}(b_1)\delta_{-\theta}^{(2)}(b_2)-\delta_{-\theta}^{(2)}(b_1)\delta_{-\theta}^{(1)}(b_2))) \\
&=-\tau_{-\theta}(b_0(\delta_{-\theta}^{(1)}(b_1)\delta_{-\theta}^{(2)}(b_2)-\delta_{-\theta}^{(2)}(b_1)\delta_{-\theta}^{(1)}(b_2))) \quad (b_0, b_1, b_2\in T^2_{-\theta}).
\end{align*}
Moreover, for $b_0, b_1\in T^2_{-\theta}$, we calculate that
\begin{align*}
\tau_\theta^{(1)}\circ\gamma_\theta(b_0, b_1)
&=\tau_\theta^{(1)}(\gamma_\theta(b_0)\delta_\theta^{(1)}(\gamma_\theta(b_1))) \\
&=\tau_\theta^{(1)}(\gamma_\theta(b_0\delta_{-\theta}^{(2)}(b_1))) \\
&=\tau_{-\theta}^{(2)}(b_0, b_1).
\end{align*} 
Similarly we have that $\tau_\theta^{(2)}\circ\gamma_\theta=\tau_{-\theta}^{(1)}$.
\end{proof}

Under the above preparation, we determine the entire cyclic cohomology of noncommutative 3-spheres $S^3_\theta$. By Theorem \ref{m-v-seq}, we have the following exact diagram:
\[
\begin{CD}
HE^{{\rm{ev}}}(S^3_\theta) @>>> HE^{{\rm{od}}}(\tori) @>-f_1^*+f_2^*>> G_\theta^1\oplus G_{-\theta}^1 \\
@Ag_1^*+g_2^*AA & &  @VVg_1^*+g_2^*V \\
G_\theta^0\oplus G_{-\theta}^0 @<<-f_1^*+f_2^*< HE^{{\rm{ev}}}(\tori) @<<< HE^{{\rm{od}}}(S^3_\theta),
\end{CD}
\]
where $G_{\pm\theta}^0=HE^{{\rm{ev}}}((D^2\times S^1)_{\pm\theta}), G_{\pm\theta}^1=HE^{{\rm{od}}}((D^2\times S^1)_{\pm\theta})$ respectively. By Proposition \ref{d2theta} and the description in its proof, the above diagram becomes the following one:
\[
\begin{CD}
HE^{{\rm{ev}}}(S^3_\theta) @>>> \C^2 @>-f_1^*+f_2^*>> \C^2 \\
@Ag_1^*+g_2^*AA & &  @VVg_1^*+g_2^*V \\
\C^2 @<<-f_1^*+f_2^*< \C^2 @<<< HE^{{\rm{od}}}(S^3_\theta).
\end{CD}
\]
We describe precisely the maps $-f_1^*+f_2^*$ to compute $HE^*(S^3_\theta)$. For the even case, we check the map
\[
-f_1^*+f_2^* : HP^{{\rm{ev}}}(\tori)=\C [\tau_\theta]\oplus\C[\tau_\theta']\to \C [\tau_\theta'\circ\widetilde{q}]\oplus\C [\tau_{-\theta}'\circ\widetilde{q}]=G^0_\theta\oplus G^0_{-\theta}.
\]
We have $f_1^*[\tau_\theta]=[\tau_\theta\circ\widetilde{q}]=0$ by the calculation in Proposition \ref{d2theta} and $f_1^*[\tau_\theta']=[\tau_\theta'\circ\widetilde{q}]$. Alternatively, it follows from Lemma \ref{trace} that $f_2^*[\tau_\theta]=[\tau_{-\theta}\circ\widetilde{q}]=0$ by the same reason for the case of $f_1^*$ and that $f_2^*[\tau_\theta']=[\tau_\theta'\circ\widetilde{q}]=-[\tau_{-\theta}'\circ\widetilde{q}]$ by Lemma \ref{trace}. On the other hand, for the odd case, we consider the map
\[
-f_1^*+f_2^* : HP^{{\rm{od}}}(\tori)=\C [\tau_\theta^{(1)}]\oplus\C [\tau_\theta^{(2)}]\to\C[\tau_\theta^{(2)}\circ\widetilde{q}]\oplus\C [\tau_{-\theta}^{(2)}\circ\widetilde{q}]=G_\theta^1\oplus G_{-\theta}^1.
\]
Similarly we compute that
\begin{align*}
f_1^*[\tau_\theta^{(2)}]&=[\tau_\theta^{(2)}\circ\widetilde{q}] \\
f_1^*[\tau_\theta^{(1)}]&=[\tau_\theta^{(1)}\circ\widetilde{q}]=0 \\
f_2^*[\tau_\theta^{(1)}]&=[\tau_\theta^{(1)}\circ\gamma_\theta\circ\widetilde{q}]=[\tau_{-\theta}^{(2)}\circ\widetilde{q}] \\
\intertext{and}
f_2^*[\tau_\theta^{(2)}]&=[\tau_\theta^{(2)}\circ\gamma_\theta\circ\widetilde{q}]=[\tau_{-\theta}^{(1)}\circ\widetilde{q}]=0
\end{align*}
by Lemma \ref{trace}.

Therefore, we have the following exact diagram:
\[
\begin{CD}
HE^{{\rm{ev}}}(S^3_\theta) @> 0 >> \C^2 @>(\lambda, \mu)\mapsto (-\mu, \lambda) >> \C^2 \\
@AAA & &  @VV 0 V \\
\C^2 @<<(\lambda, \mu)\mapsto (-\mu, -\mu)< \C^2 @<<< HE^{{\rm{od}}}(S^3_\theta),
\end{CD}
\]
by which we conclude that
\begin{align*}
HE^{{\rm{ev}}}(S^3_\theta)&\simeq \coker\{\C\oplus\C\ni (\lambda, \mu)\mapsto (-\mu, -\mu)\in\C\oplus\C\}\simeq \C, \\
HE^{{\rm{od}}}(S^3_\theta)&\simeq \ker\{\C\oplus\C\ni (\lambda, \mu)\mapsto (-\mu, -\mu)\in\C\oplus\C\}\simeq \C.
\end{align*}
This completes our computation of the entire cyclic cohomology of noncommutative 3-spheres.
\begin{theorem}
The entire cyclic cohomology of noncommutative {\rm{3}}-spheres is isomorphic to the d'Rham homology of the ordinary {\rm{3}}-spheres with complex coefficients.
\end{theorem}

\end{document}